\documentclass[12pt,centertags,oneside]{amsart}
\usepackage{amsmath,amstext,amsthm,amscd,typearea}
\usepackage{amssymb}
\usepackage{fancyhdr}
\usepackage[mathscr]{eucal}
\usepackage{mathrsfs}
\usepackage{concmath}
\usepackage{charter}
\usepackage{dsfont}  
\usepackage{hyperref}
\usepackage{color} 
\usepackage{titletoc}

\usepackage[a4paper,width=16.2cm,top=2.8cm,bottom=2.8cm]{geometry}
 
\numberwithin{equation}{section}
 
\newcommand{\field}[1]{\mathbb{#1}}

\newcommand{\R}{\field{R}}
\newcommand{\C}{\field{C}} 
\newcommand{\N}{\field{N}}

 \def\cC{\mathscr{C}}

\def\cL{\mathscr{L}}
\def\cO{\mathscr{O}}

\def\mO{\mathcal{O}}

\newcommand{\boldsym}[1]{\boldsymbol{#1}}

\newcommand\bn{\boldsym{n}}

\def\Im{{\rm Im}}

\newcommand{\cali}[1]{\mathscr{#1}}
\newcommand{\cH}{\cali{H}}
 
\DeclareMathOperator{\Ker}{Ker}

\DeclareMathOperator{\Dom}{Dom}

\DeclareMathOperator{\rank}{rank}

\DeclareMathOperator{\supp}{supp}

\DeclareMathOperator{\vol}{vol}

\newcommand{\imat}{\sqrt{-1}}

\newcommand{\om}{\omega}
\newcommand{\ol}{\overline}
\newcommand{\ddbar}{\overline\partial}
\newcommand{\dbar}{\partial}
\newtheorem{thm}{Theorem}[section]
\newtheorem{lemma}[thm]{Lemma}
\newtheorem{prop}[thm]{Proposition}
\newtheorem{cor}[thm]{Corollary}
\theoremstyle{definition}
\newtheorem{rem}[thm]{Remark}
\theoremstyle{definition}
\newtheorem{defn}[thm]{Definition}

\newcommand{\be}{\begin{eqnarray}}
\newcommand{\ee}{\end{eqnarray}}
\newcommand{\ov}{\overline}

\newcommand{\wi}{\widetilde}

\newcommand{\comment}[1]{}
\begin{document}  
\title 
{The growth of dimension of cohomology of 
semipositive line bundles on Hermitian manifolds} 
    
\author{Huan Wang}
\address{School of Mathematical Sciences, Shanghai Key Laboratory of PMMP, East China Normal University, 
	Dong Chuan Road 500, Shanghai 200241, P. R. China} 
\email {huanwang2016@hotmail.com, hwang@math.ecnu.edu.cn} 
\date{23. Oct. 2018} 
\maketitle  
  \begin{abstract}   
  	In this paper, we study the dimension of cohomology of 
	semipositive line bundles over Hermitian manifolds, and obtain an 
	asymptotic estimate for the dimension of the space of harmonic $(0,q)$-forms 
	with values in high tensor powers of a semipositive line bundle when 
	the fundamental estimate holds. As applications, we estimate the dimension 
	of cohomology of semipositive line bundles on $q$-convex manifolds, 
	pseudo-convex domains, weakly $1$-complete manifolds and complete manifolds. 
	We also obtain the estimate of cohomology on compact manifolds with semipositive 
	line bundles endowed 
	with a Hermitian metric with analytic singularities and the related vanishing theorems.  
  \end{abstract}
                 
\tableofcontents 
 \section{Introduction}  
The purpose of this paper is to prove asymptotic estimates
for the cohomology of semipositive line bundles over various
non-compact manifolds. We generalize the asymptotics obtained 
by Berndtsson \cite{BB:02} in the compact case, which in turn
refine the holomorphic Morse inequalities of Demailly \cite{Dem:85}.

Let $X$ be a compact complex manifold, let $L$ be 
a holomorphic line bundle and $E$ be a holomorphic vector
bundle on $X$. 
The Dolbeault cohomology $H^{0,q}(X,L^k\otimes E)$ plays a fundamental role
in algebraic and complex geometry and is linked to the structure 
of the manifold, cf.\ \cite{Dem, Dem:85, MM}.
If $L$ is a positive line bundle, 
$H^{0,q}(X,L^k\otimes E)=0$ for $q\geq1$ and $k$ large enough,
by the Kodaira-Serre vanishing theorem (see e.g.\ \cite [Theorem\,1.5.6]{MM})
and this can be used to prove that global holomorphic sections of $L^k\otimes E$
give a projective embedding of $X$ for large $k$ (Kodaira embedding theorem).

Assume now that $L$ is semipositive. The solution of the Grauert-Riemenschneider conjecture
\cite{GR:70}
by Siu \cite{Sil:84} and Demailly \cite{Dem:85} shows that $\dim H^{0,q}(X,L^k\otimes E)=o(k^n)$ 
as $k\to\infty$ for $q\geq1$. This can be used to show that $X$ is a Moishezon manifold,
if $(L,h^L)$ is moreover positive at least at one point.
Berndtsson \cite{BB:02} showed that we have actually $\dim H^{0,q}(X,L^k\otimes E)=O(k^{n-q})$ 
as $k\to\infty$ for $q\geq1$. 
  
We will consider first a very general situation
where we can prove the decay of the cohomology groups as above.
Let $(X,\omega)$ be a Hermitian manifold of dimension $n$. 
Let $dv_X:=\omega^n/n!$ be the volume form on $X$.
Let $(L,h^L)$ and $(E,h^E)$ be holomorphic Hermitian line bundles on $X$,
where $L$ is a line bundle.  
We denote by $(L^2_{0,q}(X,L^k \otimes E),\|\cdot\|)$ the space
of square integrable $(0,q)$-forms with values in $L^k \otimes E$ 
with respect to the $L^2$ inner product induced by the above data.
We denote by $\ddbar^E_k$ the maximal extension of the Dolbeault operator
on $L^2_{0,\bullet}(X,L^k \otimes E)$ and by 
$\ddbar^{E*}_k$ its Hilbert space adjoint.
Let $\cH^{0,q}(X,L^k \otimes E)$ be the space 
of harmonic $(0,q)$-forms with values in $L^k \otimes E$ on $X$.

For a given $0\leq q\leq n$, we say that 
\textbf{the concentration condition holds in bidegree $(0,q)$ 
for harmonic forms with values in $L^k\otimes E$ for large $k$}, 
if there exists a compact subset $K\subset X$ and $C_0>0$ such 
that for sufficiently large $k$, we have
\begin{equation}
\|s\|^2\leq C_0\int_K |s|^2 dv_X,
\end{equation}
for $s\in \Ker(\ddbar^E_k)\cap \Ker(\ddbar^{E*}_k)\cap 
L^2_{0,q}(X,L^k\otimes E)$. The set $K$ is called the exceptional compact set 
of the concentration. 
We say that \textbf{the fundamental estimate holds in 
bidegree $(0,q)$ for forms with values in $L^k\otimes E$ for large $k$}, 
if there exists a compact subset $K\subset X$ and $C_0>0$ 
such that for sufficiently large $k$, we have 
\begin{equation}
	\|s\|^2\leq C_0(\|\ddbar^E_k s\|^2+\|\ddbar^{E,*}_k s\|^2+\int_K |s|^2 dv_X),
\end{equation}
for $s\in \Dom(\ddbar^E_k)\cap \Dom(\ddbar^{E*}_k)\cap L^2_{0,q}(X,L^k\otimes E)$.
The set $K$ is called the exceptional compact set of the estimate. 
   
The first main result of this paper is an asymptotic estimate for 
$L^2$-cohomology with semipositive line bundles over Hermitian manifolds.

\begin{thm}\label{L2FE}
    Let $(X,\omega)$ be a Hermitian manifold of dimension $n$. 
    Let $(L,h^L)$ and $(E,h^E)$ be holomorphic Hermitian line bundles on $X$. 
    Assume that for some $1\leq q \leq n$ the concentration condition holds 
    in bidegree $(0,q)$ for harmonic forms with values in $L^k\otimes E$ for large $k$.
    Assume that $(L,h^L)$ is semipositive on a neighbourhood of the exceptional compact set $K$. 
	Then there exists $C>0$ such that for sufficiently large $k$, we have
	\begin{eqnarray}\label{est1}
		\dim \cH^{0,q}(X,L^k\otimes E)&\leq& Ck^{n-q}.
	\end{eqnarray}
	The same estimate also holds for reduced $L^2$-Dolbeault cohomology groups,
	\begin{equation}\label{est2}
		\dim \overline{H}^{0,q}_{(2)}(X,L^k\otimes E)\leq Ck^{n-q}.
	\end{equation}
In particular, if the fundamental estimate holds in bidegree $(0,q)$ 
for forms with values in $L^k\otimes E$ for large $k$, 
the same estimate holds for $L^2$-Dolbeault cohomology groups,
\begin{equation}
\dim H^{0,q}_{(2)}(X,L^k\otimes E)\leq Ck^{n-q}.
\end{equation}
\end{thm} 
Note that holomorphic Morse inequalities
for the $L^2$-cohomology were obtained under the 
assumption that the fundamental estimate holds
in \cite[Theorem 3.2.13]{MM}. They can only deliver
an estimate $\dim H^{0,q}_{(2)}(X,L^k\otimes E)=o(k^{n})$ as $k\to\infty$.

A geometric situation when Theorem \ref{L2FE} can be applied is
the case of a semipositive line bundle on a complete K\"ahler manifold 
which polarizes the K\"ahler metric at infinity. 

 \begin{thm}\label{complete}
 	Let $(X,\omega)$ be a complete Hermitian manifold of dimension $n$. 
	Let $(L,h^L)$ and $(E,h^E)$ be holomorphic Hermitian line bundles on $X$. 
	Assume there exists a compact subset $K\subset X$ such 
	that $\sqrt{-1}R^{(L,h^L)}=\omega$ on $X\setminus K$ 
	and $(L,h^L)$ is semipositive on $K$. 
 	
 	Then there exists $C>0$ such that for any $q\geq 1$ and sufficiently large $k$, we have
 	\begin{equation}
 	\dim H_{(2)}^{0,q}(X,L^k\otimes E)\leq Ck^{n-q}.
 	\end{equation}
 \end{thm}
Note that by \cite[Theorem 3.3.5]{MM} we have 
$$
\dim H_{(2)}^{0,0}(X,L^k\otimes K_X)\geq 
\frac{k^n}{n!}\int_X c_1(L,h^L)^n
+o(k^n)\,,\:\:k\to\infty
$$ 
in the situation of Theorem \ref{complete}.
Moreover, the Bergman kernel of $H_{(2)}^{0,0}(X,L^k\otimes K_X)$
has an asymptotic expansion in powers of $k$ on the set where
$c_1(L,h^L)>0$ by \cite[Theorem 1.7]{HM:14}.
   
We consider further the case of $q$-convex manifolds 
as application of Theorem \ref{L2FE}. For general holomorphic
Morse inequalities on $q$-convex manifolds see \cite[Theorem 3.5.8]{MM}.

\begin{thm}\label{qconvexpos}
	Let $X$ be a $q$-convex manifold of dimension $n$, 
	and let $(L,h^L), (E,h^E)$ be holomorphic Hermitian line 
	bundles on $X$. Let $\varrho$ be an exhaustion function of 
	$X$ and $K$ the exceptional set. Let $(L,h^L)$ be semipositive 
	on a sublevel set $X_c:=\{x\in X: \varrho(x)<c \}$ satisfying 
	$K\subset X_c$\,, and let $(L,h^L)$ be positive on $X_c\setminus K$. 
	
	Then there exists $C>0$ such that for any $j\geq q$ and  $k\geq 1$, we have
	\begin{equation}\label{e:sc}
	\dim H^j(X,L^k\otimes E)\leq Ck^{n-j}.
	\end{equation}
\end{thm} 

We denote here by $H^j(X,L^k\otimes E)$ the cohomology
groups of the sheaves $\mathcal{O}_X(L^k\otimes E)$ 
of holomoprhic sections of $L^k\otimes E$.
They are isomorphic to the Dolbeault cohomology
groups $H^{0,j}(X,L^k\otimes E)$.

In the case of $1$-convex 
manifolds estimate \eqref{e:sc} holds
without additional hypothesis.
\begin{thm}\label{T:1c}
	Let $X$ be a $1$-convex manifold of dimension $n$, 
	and let $(L,h^L)$ and $(E,h^E)$ be holomorphic Hermitian 
	line bundles on $X$. Let $\varrho$ be an exhaustion function 
	of $X$ and $K$ the exceptional set. Let $(L,h^L)$ be semipositive 
	on a sublevel set $X_c:=\{x\in X: \varrho(x)<c \}$ satisfying $K\subset X_c$. 	
	Then there exists $C>0$ such that for any $j\geq 1$ and $k\geq 1$
	the estimate \eqref{e:sc} holds.
\end{thm}

Similarly, we also have the estimates of cohomology over pseudoconvex 
domains and weakly $1$-complete manifolds.

\begin{thm}\label{pseudoconvex}
	Let $M\Subset X$ be a smooth (weakly) pseudoconvex domain 
	in a complex manifold $X$ of dimension $n$. Let $(L,h^L)$ and 
	$(E,h^E)$ be holomorphic Hermitian line bundles on $X$.  
	Let $(L,h^L)$ be semipositive on $M$. Moreover, assume 
	$(L,h^L)$ is positive in a neighbourhood of $bM$.
	Then there exists $C>0$ such that for any $q\geq 1$ and sufficiently large $k$, we have
	\begin{equation}\label{est4}
		\dim H^{0,q}_{(2)}(X,L^k\otimes E)\leq Ck^{n-q}.
	\end{equation}
\end{thm}

\begin{thm}\label{weakly1complete}
	Let $X$ be a weakly $1$-complete manifold of dimension $n$. 
	Let $(L,h^L)$ and $(E,h^E)$ be holomorphic Hermitian line bundles on $X$. 
	Let $(L,h^L)$ be semipositive on $X$. Moreover, assume $(L,h^L)$ 
	is positive on $X\setminus K$ for a compact subset $K\subset X$.
	Then there exists $C>0$ such that for any $q\geq 1$ and 
	sufficiently large $k$, we have
	\begin{equation}\label{est5}
		\dim H^q(X,L^k\otimes E)\leq Ck^{n-q}.
	\end{equation}
\end{thm}
   
The next result is another generalization of \cite{BB:02} 
for line bundles endowed with a Hermitian metric with analytic singularities.
Let us recall that the analogue of the Kodaira vanishing theorem
 in the case of singular metrics is the Nadel vanishing theorem
 \cite{Dem:92, Nad:90}. If $X$ is a compact K\"ahler
 manifold, $L$ and $E$ are holomorphic vector bundles with $\rank(L)=1$,
 and $h^L$ is a singular Hermitian metric such that $c_1(L,h^L)$
 is a K\"ahler current, then $H^q(X,E\otimes L^k\otimes \mathscr{J}(h^{L^k}))=0$
 for $q\geq1$ and $k$ suffiently large, where $\mathscr{J}(h^{L^k})$
 is the Nadel multiplier ideal sheaf associated to $h^{L^k}$.
 Bonavero \cite{Bona:98} obtained holomorphic Morse inequalities
 for singular Hermitian line bundles. These inequalities imply 
 that for any $q\geq 1$ we have
	\begin{equation}\label{e:bon}
	\dim H^q(X, L^k\otimes E\otimes\mathscr{J}(h^{L^k}))=o(k^n),
	\quad k\to\infty,
	\end{equation}
if the curvature $c_1(L,h)$ is semipositive on the set of points
where it is smooth. We obtain the following refinement of 	\eqref{e:bon}.
  
\begin{thm}\label{singlinebundle}
	Let $X$ be a compact complex manifold of dimension $n$ and 
	let $L$ be a holomorphic line bundle on $X$ endowed with a 
	Hermitian metric $h^L$ with analytic singularities. 
	Let $(E,h^E)$ be a holomorphic Hermitian line bundle on $X$. 
	Assume $c_1(L,h^L)\geq 0$ on the set 
	$\{x\in X: h^L~\mbox{smooth on a neighborhood of}~x \}$. 
	
	Then there exists $C>0$ such that for any 
	$q\geq 1$ and  $k\geq 1$, we have
	\begin{equation}
	\dim H^q(X, L^k\otimes E\otimes\mathscr{J}(h^{L^k}))
	\leq C k^{n-q}.
	\end{equation}
	In particular, this estimate still holds when $c_1(L,h^L)$ 
	is a positive current on $X$.
\end{thm}  

\begin{rem}
	Under the hypothesis of Theorem \ref{singlinebundle}, 
	if $c_1(L,h^L)\geq 0$ on the set 
	$\{x\in X: h^L~\mbox{smooth on a neighborhood of}~x \}$ 
	and positive at least at one point, then there exist $C_1>0$ and $C_2>0$ 
	such that for sufficiently large $k$, we have
	\begin{equation}
	C_1k^n\leq \dim H^{0}(X, L^k\otimes E\otimes \mathscr{J}(h^{L^k}))\leq C_2k^n.
	\end{equation}
	Thus $L$ is big and $X$ is Moishezon.
\end{rem}

 The last results are analytic singular versions of the
 vanishing theorem for partially positive line bundles 
 \cite[Ch.VII. (5.1) Theorem]{Dem}
 and Siu's vanishing theorem \cite[Theorem 1 (Page 175)]{Siu:85}
 	 (\cite[Appendix, Theorem 6 and Theorem 3 (ii)]{Rie:71} implies that, for a compact complex manifold $X$ and a holomorphic Hermitian line bundle $(L,h^L)$ on $X$, if $X$ is Moishezon and $c_1(L,h^L)\geq 0$ on $X$ and $c_1(L,h^L)>0$ at least at one point, then 
$ 	H^q(X,K_X\otimes L)=0$
 for all $q\geq 1$.	Furthermore, Siu's resolution \cite[Theorem 1 (Page 175)]{Siu:85} of Grauert-Riemenschneider conjecture
shows that the Moishezon assumption is unnecessary, and we called it Siu's vanishing theorem, see also \cite[Ch.VII. (3.5) Theorem]{Dem}). 

\begin{thm}\label{singAG62}
	Let $X$ be a compact manifold of dimension $n$, 
	let $L$ and $E$ be holomorphic vector bundles with $\rank(L)=1$. 
	Let $h^L$ be the Hermitian metric on $L$ with analytic singularities. 
	Assume $c_1(L,h^L)$ has at least $n-s+1$ positive eigenvalues on 
	the set $\{x\in X: h^L~\mbox{smooth on a neighborhood of}~x \}$.

 Then, for any $q\geq s$ and sufficiently large $k$, we have
	\begin{equation}
	H^q(X,E\otimes L^k\otimes \mathscr{J}(h^{L^k}))=0.
	\end{equation}
	Therefore, it follows that
	\begin{equation}\label{eqsum}
	\sum_{q=0}^{s-1}(-1)^q\dim H^q(X,E\otimes L^k\otimes \mathscr{J}(h^{L^k}))
	=\frac{k^n}{n!}\int_{X(\leq s-1)}c_1(L,h^L)^n+o(k^n).
	\end{equation}
\end{thm}  
   
Note that for a Hermitian metric with analytic singularities on line bundles over connected compact complex manifolds (see Definition \ref{defn_singmetric}), the non-negative rational number $c$ does not depend on the local form (\ref{singlocalweights}) of the weight $\varphi$. 
\begin{thm}\label{singSiu84}
	Let $X$ be a compact manifold of dimension $n$, let $L$ be a holomorphic line bundle. Let $h^L$ be the Hermitian metric on $L$ with analytic singularities. Assume $c_1(L,h^L)\geq 0$ on the set $\{x\in X: h^L~\mbox{smooth on a neighborhood of}~x \}$ and $c_1(L,h^L)>0$ at least at one point. 
	
	Then, for any $q\geq 1$, $k\geq 1$ and $m\in \N\setminus\{ 0 \}$ satisfying $mc\in \N$ with non-negative rational number $c$ in (\ref{singlocalweights}), we have 
	\begin{equation}
	H^q(X,K_X\otimes L^{km}\otimes \mathscr{J}(h^{L^{km}}))=0.
	\end{equation}
\end{thm}

In this paper we consider the cohomology spaces on general 
(possibly non-compact) complex manifolds with semipositive line bundles. 
With the fundamental estimates fulfilled, Theorem \ref{L2FE} give an 
estimate of $L^2$-Dolbeaut cohomology on arbitrary complex manifolds. 
On one hand, it generalizes \cite{BB:02} to general complex manifolds 
in the context of $L^2$-cohomology; on the other hand, it leads to the refinement 
of the estimates for $q$-convex manifolds, pseudoconvex domains, 
weakly $1$-convex manifolds, complete manifolds, 
and semipositive line bundle endowed metric with analytic singularities, 
see Theorem \ref{qconvexpos}, \ref{pseudoconvex}, 
\ref{weakly1complete}, \ref{complete}, and \ref{singlinebundle}, respectively. 
Note also that the magnitude $k^{n-q}$ in Theorem \ref{L2FE}, Theorem \ref{qconvexpos}, 
Theorem \ref{weakly1complete}, Theorem \ref{complete}, 
Theorem \ref{singlinebundle} and Theorem \ref{coveringthm} 
cannot be improved in general, see \cite[Proposition 4.2]{BB:02}.

Our paper is organized in the following way. In Section \ref{pre_section} 
we introduce the notations and recall the necessary facts. 
In Section \ref{L2_section}, we give an asymptotic estimate for $L^2$-cohomology 
with semipositive line bundles on Hermitian manifolds, 
which is a uniform approach to consider semipositive line bundles 
on complex manifolds. As applications, we obtain the estimate of 
growth of dimension in certain possibly non-compact complex manifolds. 
In additional, we revisited the compact and covering manifolds in this context. 
In Section \ref{sectionsing}, we prove the estimate of cohomology still 
holds when the Hermitian metric of the line bundle has analytic singularities. 
Meanwhile, two vanishing theorems of line bundles are given in the setting of 
analytic singular Hermitian metric. The techniques and formulations are mainly 
based on Berndtsson \cite{BB:02}, Ma-Marinescu \cite{MM} and \cite{Wh:16}. 

\section{Preliminaries}\label{pre_section}

\subsection{$L^2$-cohomology}
Let $(X, \omega)$ be a Hermitian manifold of dimension $n$ and $(F, h^F)$ be a holomorphic Hermitian vector bundle over $X$. Let $\Omega^{p,q}(X, F)$ be the space of smooth $(p,q)$-forms on $X$ with values in $F$ for $p,q\in \N$. If the $\rank(F)=1$, the curvature of $(F, h^F)$ is defined by $R^F=\ddbar\dbar \log|s|^2_{h^{F}}$ for any local holomorphic frame $s$, then the Chern-Weil form of the first Chern class of $F$ is $c_1(F, h^F)=\frac{\sqrt{-1}}{2\pi}R^F$, which is a real $(1,1)$-form on $X$. The volume form is given by $dv_{X}:=\omega_n:=\frac{\omega^n}{n!} $. We use the notion of positive $(p,p)$-form given by \cite[Chapter III, \S 1, (1.1) (1.2) (1.5) (1.7)]{Dem}. If a $(p,p)$-form $T$ is positive, we write $T\geq 0$. 
\begin{defn}
We say a holomorphic Hermitian line bundle $(L,h^L)$ is semipositive on $X$, if $c_1(L,h^L)$ is positive semi-definite on $X$, equivalently $c_1(L,h^L)\geq 0$. For simplifying notations, we also denote  $L\geq 0$. 	
\end{defn} 

Let $\Omega^{p,q}_0(X, F)$ be the subspace of 
$\Omega^{p,q}(X, F)$ consisting of elements with compact support.
The $L^2$-scalar product on $\Omega^{p,q}_0(X, F)$ given by 
\begin{equation}\label{e:sp}
\langle s_1,s_2 \rangle=\int_X \langle s_1(x), s_2(x) \rangle_h dv_X(x)
\end{equation}
where $\langle\cdot,\cdot\rangle_h:=
\langle\cdot,\cdot\rangle_{h^F,\omega}$ 
is the pointwise Hermitian inner product induced by $\omega$ and $h^F$. 
We denote by $L^2_{p,q}(X, F)$, the $L^2$ completion of $\Omega^{p,q}_0(X, F)$.

Let $\ddbar^{F}: \Omega_0^{p,q} (X, F)\rightarrow L^2_{p,q+1}(X,F) $ be the Dolbeault operator and let  $ \ddbar^{F}_{\max} $ be its maximal extension (see \cite[Lemma 3.1.1]{MM}). From now on we still denote the maximal extension by $ \ddbar^{F} :=\ddbar^{F}_{\max} $ and  the associated Hilbert space adjoint by $\ddbar^{F*}:=\ddbar^{F*}_H:=(\ddbar^{F}_{\max})_H^*$ for simplifying the notations. Consider the complex of closed, densely defined operators
$L^2_{p,q-1}(X,F)\xrightarrow{\ddbar^{F}}L^2_{p,q}(X,F)\xrightarrow{\ddbar^{F}} L^2_{p,q+1}(X,F)$,
then $(\ddbar^{F})^2=0$. By \cite[Proposition 3.1.2]{MM}, the operator defined by
\begin{eqnarray}\label{eq40}\nonumber
\Dom(\square^{F})&=&\{s\in \Dom(\ddbar^{F})\cap \Dom(\ddbar^{F*}): 
\ddbar^{F}s\in \Dom(\ddbar^{F*}),~\ddbar^{F*}s\in \Dom(\ddbar^{F}) \}, \\ 
\square^{F}s&=&\ddbar^{F} \ddbar^{F*}s+\ddbar^{F*} \ddbar^{F}s \quad \mbox{for}~s\in \Dom(\square^{F}),
\end{eqnarray}
is a positive, self-adjoint extension of Kodaira Laplacian, called the Gaffney extension.

\begin{defn}
	The space of harmonic forms $\cH^{p,q}(X,F)$ is defined by 
	\begin{equation}\label{eq45}
	\cH^{p,q}(X,F):=\Ker(\square^{F})=\{s\in \Dom(\square^{F})\cap L^2_{p,q}(X, F): \square^{F}s=0 \}.
	\end{equation}
	The $q$-th reduced $L^2$-Dolbeault cohomology is defined by 
	\begin{equation}\label{eq46}
	\overline{H}^{0,q}_{(2)}(X,F):=\dfrac{\Ker(\ddbar^{F})\cap  L^2_{0,q}(X,F) }{[ \Im( \ddbar^{F}) \cap L^2_{0,q}(X,F)]},
	\end{equation}
	where $[V]$ denotes the closure of the space $V$. The $q$-th (non-reduced) $L^2$-Dolbeault cohomology is defined by 
	\begin{equation}
	H^{0,q}_{(2)}(X,F):=\dfrac{\Ker(\ddbar^{F})\cap  L^2_{0,q}(X,F) }{ \Im( \ddbar^{F}) \cap L^2_{0,q}(X,F)}.
	\end{equation}
	\end{defn}  
	
	According to the general regularity theorem of elliptic operators (cf.\ \cite[Theorem A.3.4]{MM}), 
	$s\in \cH^{p,q}(X,F) $ implies $s\in\Omega^{p,q}(X,F)$. By weak Hodge decomposition (cf.\ \cite[(3.1.21) (3.1.22)]{MM}), 
	we have a canonical isomorphism  
	\begin{equation}\label{eq47}
	\overline{H}^{0,q}_{(2)}(X,F)\cong \cH^{0,q}(X,F)
	\end{equation} for any $q\in \N$, which associates to each cohomology class its unique harmonic representative. The $q$-th cohomology of the sheaf of holomorphic sections of $F$ is isomorphic to the the $q$-th Dolbeault cohomology, $H^q(X,F)\cong H^{0,q}(X,F)$.
	
	For a given $0\leq q \leq n$, we say the fundamental estimate holds in bidegree $(0,q)$ for forms with values in $F$, if there exists a compact subset $K\subset X$ and $C>0$ such that 
	\begin{equation}
	\|s\|^2\leq C(\|\ddbar^F s\|^2+\|\ddbar^{F*}\|^2+\int_K|s|^2dv_X),
	\end{equation}
	for $s\in \Dom(\ddbar^F)\cap\Dom(\ddbar^{F,*})\cap L^2_{0,q}(X,F)$. $K$ is called the exceptional compact set of the estimate. If the fundamental estimate holds in bidegree $(0,q)$ for forms with values in $F$, the reduced and non-reduced $L^2$-Dolbeault cohomology coincide, see \cite[Theorem 3.1.8]{MM}.
	
	For a given $0\leq q\leq n$, we say that the concentration condition holds in bidegree $(0,q)$ for harmonic forms with values in $F$, if there exists a compact subset $K\subset X$ and $C>0$ such that 
	\begin{equation}
	\|s\|^2\leq C\int_K |s|^2 dv_X,
	\end{equation}
	for $s\in \Ker(\ddbar^F)\cap \Ker(\ddbar^{F*})\cap L^2_{0,q}(X,F)$.
	The compact set $K$ is called the exceptional compact set of the concentration. 
	Note if the fundamental estimate holds in bidegree $(0,q)$ 
	for forms with values in $F$, the concentration condition holds in bidegree $(0,q)$ 
	for harmonic forms with values in $F$.
	
	\subsection{$q$-convex complex manifolds and $\Gamma$-coverings}
	\begin{defn}[{\cite{AG:62}}]
		A complex manifold $X$ of dimension $n$ is called $q$-convex if there exists a smooth function $\varrho\in \cC^\infty(X,\R)$ such that the sublevel set $X_c=\{ \varrho<c\}\Subset X$ for all $c\in \R$ and the complex Hessian $\dbar\ddbar\varrho$ has $n-q+1$ positive eigenvalues outside a compact subset $K\subset X$. Here $X_c\Subset X$ means that the closure $\overline{X}_c$ is compact in $X$. We call $\varrho$ an exhaustion function and $K$ exceptional set. We say $X$ is $q$-complete if $K=\emptyset$ in additional.	
	\end{defn} 
	
	\begin{defn}  
		A complex manifold $X$ of dimension $n$ is called a 
		$q$-convex manifold with a plurisubharmonic exhaustion function near the exceptional set, 
		if there exists a compact subset $K\subset X$ and
		a smooth function $\varrho\in \cC^\infty(X,\R)$ such that the 
		sublevel set $X_c:=\{\varrho< c  \}\Subset X$ for all $c\in \R$, and $\sqrt{-1}\dbar\ddbar\varrho$ has at least $n-q+1$ positive eigenvalues 
		on $X\setminus K$ and $\sqrt{-1}\dbar\ddbar\varrho\geq 0$ on 
		$X_c\setminus K$ for some $X_c$ with $K\subset X_c$.	
	\end{defn}
	
Let $M$ be a relatively compact domain with smooth boundary $bM$ in a complex manifold $X$. Let $\rho\in \cC^\infty(X,\R)$ such that $M=\{ x\in X: \rho(x)<0 \}$ and $d\rho\neq 0$ on $bM=\{x\in X: \rho(x)=0\}$. We denote the closure of $M$ by $\overline{M}=M\cup bM$. We say that $\rho$ is a defining function of $M$. Let $T^{(1,0)}bM:=\{ v\in T_x^{(1,0)}X: \dbar\varrho(v)=0 \}$ be the analytic tangent bundle to $bM$ at $x\in bM$. The Levi form of $\rho$ is the $2$-form $\cL_\rho:=\dbar\ddbar\rho\in \cC^\infty(bM, T^{(1,0)*}bM\otimes T^{(0,1)*}bM)$.

\begin{defn}
	A relatively compact domain $M$ with smooth boundary $bM$ in a 
	complex manifold $X$ is called strongly (resp.\ (weakly)) pseudoconvex 
	if the Levi form $\cL_\rho$ is positive definite (resp.\ semidefinite).
\end{defn}

Note that any strongly pseudoconvex domain is $1$-convex.

\begin{defn}
	A complex manifold $X$ is called weakly $1$-complete if there exists a smooth plurisubharmonic function $\varphi\in \cC^\infty(X,\R)$ such that $\{x\in X: \varphi(x)<c\}\Subset X$ for any $c\in \R$. $\varphi$ is called an exhaustion function.
\end{defn}
Note that any $1$-convex manifold is weakly $1$-complete.

	\begin{defn}
		A Hermitian manifold $(X,\omega)$ is called complete, if all geodesics are defined for all time for the underlying Riemannian manifold. 
	\end{defn}

If $(X,\omega)$ is complete, for arbitrary holomorphic Hermitian vector bundle $(F,h^F)$ on $X$, $\Omega_0^{0,\bullet}(X,F)$ is dense in $\Dom(\ddbar^F)$, $\Dom(\ddbar^{F*}_H)$ and $\Dom(\ddbar^F)\cap \Dom(\ddbar^{F*}_H)$ in the graph-norms of $\ddbar^F$, $\ddbar^{F*}_H$ and $\ddbar^E+\ddbar^{E*}_H$ respectively, see \cite[Lemma 3.3.1 (Andreotti-Vesentini), Corollary 3.3.3]{MM}. Here the graph-norm is defined by $\|s\|+\|Rs\|$ for $s\in \Dom(R)$. 
	
	\begin{defn}
		Let $(X,\omega)$ be a Hermitian manifold of dimension $n$ on which a discrete 
		group $\Gamma$ acts holomorphically, freely and properly such that $\omega$ is a $\Gamma$-invariant Hermitian 
		metric and the quotient $X/\Gamma$ is compact. We say $X$ is a $\Gamma$-covering manifold.
	\end{defn}	
 
\subsection{Kodaira Laplacian with $\ddbar$-Neumann boundary conditions}
Let $(X,\omega)$ be a Hermitian manifold of dimension $n$ and $(F,h^F)$ be a holomorphic Hermitian vector bundles over $X$. Let $M$ be a relatively compact domain in $X$. Let $\rho$ be a defining function of $M$ satisfying $M=\{ x\in X: \rho(x)<0 \}$ and $|d\rho|=1$ on $bM$, where the pointwise norm $|\cdot|$ is given by $g^{TX}$ associated to $\omega$.

Let $e_{\bn} \in TX$ be the inward pointing unit normal at $bM$ and $e_{\bn}^{(0,1)}$ its projection on $T^{(0,1)}X$. In a local orthonormal frame $\{ w_1,\cdots,\omega_n \}$ of $T^{(1,0)}X$, we have $e_{\bn}^{(0,1)}=-\sum_{j=1}^n w_j(\rho)\ov w_j$. Let $B^{0,q}(X,F):=\{ s\in \Omega^{0,q}(\ov M, F): i_{e_{\bn}^{(0,1)}}  s=0 ~\mbox{on}~bM \}$. We have $B^{0,q}(M,F)=\Dom(\ddbar_H^{F*})\cap \Omega^{0,q}(\overline{M},F)$ and the Hilbert space adjoint $\ddbar_H^{F*}$ of $\ddbar^F$ coincides with the formal adjoint $\ddbar^{F*}$ of $\ddbar^F$ on $B^{0,q}(M,F)$, see \cite[Proposition 1.4.19]{MM}. We consider the operator $\square_N s=\ddbar^{F}\ddbar^{F*}s+\ddbar^{F*}\ddbar^{F}s$ for $s\in \Dom(\square_N):=\{s\in B^{0,q}(M,F): \ddbar^Fs\in B^{0,q+1}(M,F) \}$. The Friedrichs extension of $\square_N$ is a self-adjoint operator and is called the Kodaira Laplacian with $\ddbar$-Neumann boundary conditions, which coincides with the Gaffney extension of the Kodaira Laplacian, see \cite[Proposition 3.5.2]{MM}. Note $\Omega^{0,\bullet}(\overline{M},F)$ is dense in $\Dom(\ddbar^F)$ in the graph-norms of $\ddbar^F$, and $B^{0,\bullet}(M,F)$ is dense in $\Dom(\ddbar^{F*}_H)$ and in $\Dom(\ddbar^F)\cap \Dom(\ddbar^{F*}_H)$ in the graph-norms of $\ddbar^{F*}_H$ and $\ddbar^E+\ddbar^{E*}_H$, respectively, see \cite[Lemma 3.5.1]{MM}. Here the graph-norm is defined by $\|s\|+\|Rs\|$ for $s\in \Dom(R)$. 

\subsection{Hermitian metric with analytic singularities on line bundles} 
\begin{defn} \label{defn_singmetric}
	Let $X$ be a connected compact complex manifold of dimension $n$ and $L$ a holomorphic line bundle on $X$. On $L$ we say $h^L$ is a Hermitian metric with analytic singularities, if there exists a smooth hermitian metric $h^L_0$ and a function $\varphi\in L^1_{loc}(X,\R)$ with locally
	\begin{equation}\label{singlocalweights}
	\varphi=\frac{c}{2}\log(\sum_{j\in J}|f_j|^2)+\psi,
	\end{equation}
	where $J$ is at most countable, $c$ is a non-negative rational number, $f_j$ are non-zero holomorphic functions and $\psi$ is a smooth function, such that $h^L=h^L_0e^{-2\varphi}$. 
\end{defn} 
 
For any local holomorphic frame $e_L$ of $L$, $h_0^L(e_L,e_L)=e^{-2\psi_0}$, where $\psi_0$ is smooth, thus the local weights of $h^L$ is given by
$\frac{c}{2}\log(\sum_J|f_j|^2)+(\psi+\psi_0).$ Because locally $\psi+\psi_0$ is smooth, we use (\ref{singlocalweights}) to represent the local weight of $h^L$ for simplifying notations.

For $\varphi\in L^1_{loc}(X,\R)$, the Nadel multiplier ideal sheaf $\mathscr{J}(\varphi)$ is the ideal subsheaf of germs of holomorphic functions $f\in \cO_{X,x}$ such that $|f|^2e^{-2\varphi}$ is integrable with respect to the Lebesgue measure in local coordinates near $x$. We define $\mathscr{J}(h^L):=\mathscr{J}(\varphi)$, which does not depend on the choice of $\varphi$, see \cite{Dem:92} or \cite[Definition 2.3.13]{MM}. 

\section{Asymptotic estimate for $L^2$-cohomology with semipositive line bundles}\label{L2_section}
Let $(X,\omega)$ be a Hermitian manifold of dimension $n$ and let
$(L,h^L)$ and $(E,h^E)$ be holomorphic Hermitian line bundles over $X$. Let $\cH^{0,q}(X,L^k\otimes E)$ be the space of harmonic $(0,q)$-forms with values in $L^k\otimes E$. Let $\{s^k_j\}_{j\geq1}$ be an orthonormal basis and denote by $B_k^q$
the Bergman density function defined by
\begin{equation}\label{e:Bergfcn}
B_k^q(x)=\sum_{j\geq 1}|s^k_j(x)|_{h_k,\omega}^2\,,
\;x\in X,
\end{equation}
where $|\cdot|_{h_k,\omega}$ is the pointwise norm of a form, see \cite {Wh:16}. The function \eqref{e:Bergfcn} 
is well-defined by an adaptation of 
\cite[Lemma 3.1]{CM} to form case.  
 By replacing $E\bigotimes\Lambda^n (T^{(1,0)}X)$ for $E$ in $\cH^{n,q}(X,L^k\otimes E)$, we can rephrase \cite[Theorem 1.1]{Wh:16}  as follows.

\begin{thm}\label{Localmain}
	Let $(X,\omega)$ be a Hermitian manifold of dimension $n$ and let
	$(L,h^L)$ and $(E,h^E)$ be holomorphic Hermitian line bundles over $X$. 
	Let $K\subset X$ be a compact subset and assume that $(L,h^L)$ is semipositive on a neighborhood
	of $K$.
	
	Then there exists  $C>0$ depending on the compact set $K$, the metric $\omega$ 
	and the bundles $(L,h^L)$ and $(E,h^E)$, such that for any $x\in K$, $k\geq 1$ and $q\geq 1$,
	\begin{equation}
	B^q_k(x) \leq C k^{n-q}\,,
	\end{equation}
	where $B^q_k(x)$ is the Bergman density function \eqref{e:Bergfcn} of harmonic 
	$(0,q)$-forms with values in $L^k\otimes E$. 
\end{thm}
 
A general result on asymptotic estimate for $L^2$-cohomology with semipositive line bundles over Hermitian manifolds follows immediately.

\begin{thm}[Theorem \ref{L2FE}]
	Let $(X,\omega)$ be a Hermitian manifold of dimension $n$. Let $(L,h^L)$ and $(E,h^E)$ be holomorphic Hermitian line bundles on $X$. Let $1\leq q \leq n$. Assume the concentration condition holds in bidegree $(0,q)$ for harmonic forms with values in $L^k\otimes E$ for large $k$.
	Moreover, assume
	$(L,h^L)$ is semipositive on a neighbourhood of the exceptional set $K$. 
	Then there exists $C>0$ such that for sufficiently large $k$ we have
	\begin{eqnarray}
	\dim \cH^{0,q}(X,L^k\otimes E)&\leq& Ck^{n-q}.
	\end{eqnarray}
	The same estimate also holds for reduced $L^2$-Dolbeault cohomology groups,
	\begin{equation}
	\dim \overline{H}^{0,q}_{(2)}(X,L^k\otimes E)\leq Ck^{n-q}.
	\end{equation}
	In particular, if the fundamental estimate holds in bidegree $(0,q)$ for forms with values in $L^k\otimes E$ for large $k$, the same estimate holds for $L^2$-Dolbeault cohomology groups
	\begin{equation}
	\dim H^{0,q}_{(2)}(X,L^k\otimes E)\leq Ck^{n-q}.
	\end{equation}
\end{thm}

\begin{proof}
By Theorem \ref{Localmain} and the concentration condition, we have
\begin{eqnarray}
\dim \ov H^{(0,q)}_{(2)}
&=&\dim \cH^{0,q}(X,L^k\otimes E)\\
&=&\sum_{j\geq 1} \|s^k_j\|^2\leq C_0\int_K B^q_k(x)dv_X\leq C_0C\vol(K)k^{n-q}
\end{eqnarray}
for sufficiently large $k$. Note that $H^{0,q}_{(2)}(X,F)= \ov H^{0,q}_{(2)}(X,F)$ and the dimension is finite, when the fundamental estimate holds in bidegree $(0,q)$ for forms with values in a holomorphic Hermitian vector bundle $(F,h^F)$ by \cite[Theorem 3.1.8]{MM}. 
\end{proof}
      
\subsection{$q$-convex manifolds}
\subsubsection{Exhaustion functions with the plurisubharmonic near the exceptional set}

In this section we prove the following
general result about the growth of the
cohomology of $q$-convex manifolds.  
\begin{thm}\label{qconvexpshthm}
	Let $X$ be a $q$-convex manifold of dimension $n$, 
	and let $(L,h^L), (E,h^E)$ be holomorphic Hermitian line bundles on $X$. 
	Let $\varrho$ be an exhaustion function of $X$ and $K$ the 
	exceptional set. Let $(L,h^L)$ be semipositive on a sublevel set 
	$X_c:=\{x\in X: \varrho(x)<c \}$ satisfying 
	$K\subset X_c$, and let $\sqrt{-1}\dbar\ddbar\varrho\geq 0$ on $X_c\setminus K$.
	Then there exists $C>0$ such that for any $j\geq q$ and  $k\geq 1$, we have
	\begin{equation}
	\dim H^j(X,L^k\otimes E)\leq Ck^{n-j}.
	\end{equation}
\end{thm} 

Let $X$ be a $q$-convex manifold of dimension $n$, let $\varrho$ be a plurisubharmonic exhaustion function of $X$ and $K$ the exceptional set. By the definition, $\varrho\in \cC^\infty(X,\R)$ satisfies  $X_c:=\{ \varrho<c\}\Subset X$ for all $c\in \R$,  $\sqrt{-1}\dbar\ddbar\varrho$ has $n-q+1$ positive eigenvalues on $X\setminus K$. In this section, we fix real numbers $u_0, u$ and $v$ satisfying $u_0<u<c<v$ and $K\subset X_{u_0}$.  
 
Let $(L,h^L), (E,h^E)$ be holomorphic Hermitian line bundles on $X$. We have that the fundamental estimate holds in bidegree $(0,j)$ for forms with values in $L^k\otimes E$ for large $k$ and each $q\leq j\leq n$ on $X_c$ when $X$ is a $q$-convex manifold, see Proposition \ref{1coxFE}, which was obtained in \cite[Theorem 3.5.8]{MM} for the proof of Morse inequalities on $q$-convex manifold. For the sake of completeness, we prove it here. 

Firstly, we choose now a Hermitian metric on a $q$-convex manifold $X$.

\begin{lemma}(\cite[Lemma 3.5.3]{MM}). \label{lowbd_rho_lem}
	For any $C_1>0$ there exists a metric $g^{TX}$ (with Hermitian form $\omega$) on $X$ such that for any $j\geq q$ and any holomorphic Hermitian vector bundle $(F,h^F)$ on $X$,
	\begin{equation}
	\langle (\dbar\ddbar\varrho)(w_l,\ol{w}_k)\ol{w}^k\wedge i_{\ol{w}_l}s,s \rangle_h\geq C_1|s|^2, \quad s\in \Omega^{0,j}_0(X_v\setminus \ov X_u,F),
	\end{equation}
	where $\{ w_l \}_{l=1}^n$ is a local orthonormal frame of $T^{(1,0)}X$ with dual frame $\{ w^l\}_{l=1}^n$ of $T^{(1,0)*}X$.
\end{lemma}  

Now we consider the $q$-convex manifold $X$ associated with the metric $\omega$ obtained above as a Hermitian manifold $(X,\om)$. Note for arbitrary holomorphic vector bundle $F$ on a relatively compact domain $M$ in $X$, the Hilbert space adjoint $\ddbar_H^{F*}$ of $\ddbar^F$ coincides with the formal adjoint $\ddbar^{F*}$ of $\ddbar^F$ on $B^{0,j}(M,F)=\Dom(\ddbar_H^{F*})\cap \Omega^{0,j}(\overline{M},F)$, $1\leq j\leq n$. So we simply use the notion $\ddbar^{F*}$ on $B^{0,j}(M,F)$, $1\leq j\leq n$. 

Secondly, we modify Hermitian metric $h^L_\chi$ on $L$ and show the fundamental estimate fulfilled.
Let $\chi(t)\in\cC^\infty(\R)$ such that $\chi'(t)\geq 0$, $\chi''(t)\geq 0$. We define a Hermitian metric $h^{L^k}_\chi:=h^{L^k}e^{-k\chi(\varrho)}$ on $L^k$ for each $k\geq 1$ and set $L^k_\chi:=(L^k,h^{L^k}_\chi)$. Thus 
\begin{equation}
R^{L^k_\chi}=kR^{L_\chi}=kR^L+k\chi'(\varrho)\dbar\ddbar\varrho+k\chi''(\varrho)\dbar\varrho\wedge\ddbar\varrho.
\end{equation}

\begin{lemma}(\cite[(3.5.19)]{MM})\label{keylem}
	There exists $C_2>0$ and $C_3>0$ such that, if $\chi'(\varrho)\geq C_3$ on $X_v\setminus \overline{X}_u$, then
	\begin{equation}
	\|s\|^2\leq \frac{C_2}{k}( \|\ddbar^E_k s\|^2+\|\ddbar^{E*}_k s\|^2 )
	\end{equation}   
	for $s\in B^{0,j}(X_c,L^k\otimes E)$ with $\supp(s)\subset X_v\setminus \overline{X}_u$, $j\geq q$ and $k\geq 1$,
	where the $L^2$-norm $\|\cdot\|$ is given by $\omega$, $h^{L^k}_\chi$ and $h^E$ on $X_c$.
\end{lemma}

\begin{lemma}\label{K'lem}
	Let $\epsilon>0$ satisfying $X_{c+\epsilon}\setminus \ov{X}_{c-\epsilon}:=\{ c-\epsilon<\varrho<c+\epsilon \}\Subset X_v\setminus \ov{X}_u$. Let $\phi\in \cC^\infty_0(X_v,\R)$ with $\supp(\phi)\subset X_v\setminus \ov{X}_u$ such that $0\leq \phi \leq 1$ and $\phi=1$ on $X_{c+\epsilon}\setminus \ov{X}_{c-\epsilon}$. Let $K':=\ov{X}_{c-\epsilon}:={\{\varrho\leq c-\epsilon\}}$.
	Then, for any $s\in B^{0,p}(X_c,L^k\otimes E)$, $1\leq p\leq n$, we have 
	\begin{equation}
	\|\phi s\|^2\geq \|s\|^2- \int_{K'}|s|dv_X,
	\end{equation} 
	where the Hermitian norm $|\cdot|$ and the $L^2$-norm $\|\cdot\|$ are given by $\omega$, $h^{L^k}_\chi$ and $h^E$ on $X_c$.
\end{lemma}
\begin{proof}
	For $s\in B^{0,p}(X_c,L^k\otimes E)=\Dom(\ddbar^{E*}_k)\cap\Omega^{0,p}(\ov{X}_c,L^k\otimes E)$, $\phi s\in \Omega^{0,p}(\ov{X}_c,L^k\otimes E)$ and $i_{e_{\bn}^{(0,1)}}(\phi s)=i_{e_{\bn}^{(0,1)}}( s)=0$ on $bX_c$ by $\phi s=s$ on the neighbourhood $X_{c+\epsilon}\setminus \ov{X}_{c-\epsilon}$ of $bX_c$. Thus $\phi s \in B^{0,p}(X_c,L^k\otimes E)$ with $\supp(\phi s)\subset X_v\setminus \ov{X}_u$,
	\begin{eqnarray}\nonumber
	\|\phi s\|^2&=&\int_{X_c}|\phi s|^2dv_X=\int_{X_c\setminus \ov{X}_u}|\phi s|^2dv_X
	=\int_{\{c-\epsilon<\varrho<c\}}|\phi s|^2 dv_X+\int_{u<\varrho\leq c-\epsilon}|\phi s|^2 dv_X\\\nonumber
	&=&\int_{\{c-\epsilon<\varrho<c\}}|s|^2 dv_X+\int_{\{u<\varrho\leq c-\epsilon\}}|\phi s|^2 dv_X\\
	&\geq&\int_{X_c\setminus \ov{X}_{c-\epsilon}}|s|^2 dv_X
	=\|s\|^2-\int_{K'}|s|^2 dv_X.
	\end{eqnarray}
\end{proof} 
   
\begin{lemma}
	Let $\phi$ be in Lemma \ref{K'lem}, and let $\xi:=1-\phi$ and $C_1:=\sup_{x\in X_c}|d\xi(x)|_{g^{T^*X}}^2>0$. Then, for any $s\in B^{0,p}(X_c,L^k\otimes E)$, $1\leq p\leq n$, and $k\geq 1$, we have 
	\begin{equation}
	\|\ddbar^E_k(\xi s)\|^2+\|\ddbar^{E*}_k(\xi s)\|^2
	\leq \frac{3}{2}(\|\ddbar^E_k s\|^2+\|\ddbar^{E*}_k s\|^2)+6C_1\|s\|^2,
	\end{equation}
	\begin{equation}\label{cuteq_k}
	\frac{1}{k}(\|\ddbar^E_k (\phi s)\|^2+\|\ddbar^{E*}_k(\phi s)\|^2)
	\leq \frac{5}{k}(\|\ddbar^E_k s\|^2+\|\ddbar^{E*}_k s\|^2)+\frac{12C_1}{k}\|s\|^2,
	\end{equation}
	where the $L^2$-norm $\|\cdot\|$ is given by $\omega$, $h^{L^k}_\chi$ and $h^E$ on $X_c$.
\end{lemma}
\begin{proof} 
	The first inequality follows from \cite[(3.2.8)]{MM}. For simplifying notations, we use $\ddbar$ and $\ddbar^*$ instead of $\ddbar^E_k$ and $\ddbar^{E*}_k$ respectively. From $\frac{1}{2}\|\ddbar(\phi s)\|^2-\|\ddbar s\|^2\leq \|\ddbar s-\ddbar(\phi s)\|^2$, $\frac{1}{2}\|\ddbar^*(\phi s)\|^2-\|\ddbar^* s\|^2\leq \|\ddbar^* s-\ddbar^*(\phi s)\|^2$
	and the first inequality, we have
	\begin{equation}
	\frac{1}{2}(\|\ddbar(\phi s)\|^2+\|\ddbar^*(\phi s)\|^2)\leq \frac{5}{2}(\|\ddbar s\|^2+\|\ddbar^* s\|^2)+6C_1\|s\|^2,
	\end{equation} 
	thus the second inequality follows.
\end{proof}

\begin{prop}\label{1coxFE}
	Let $X$ be a $q$-convex manifold of dimension $n$ with the exceptional set $K\subset X_c$. Then there exists a compact subset $K'\subset X_c$ and $C_0>0$ such that for sufficiently large $k$, we have  
	\begin{equation}
	\|s\|^2\leq \frac{C_0}{k}(\|\ddbar^E_ks\|^2+\|\ddbar^{E*}_{k,H}s\|^2)+C_0\int_{K'} |s|^2 dv_X
	\end{equation}
	for any $s\in \Dom(\ddbar^E_k)\cap \Dom(\ddbar^{E*}_{k,H})\cap L^2_{0,j}(X_c,L^k\otimes E)$ and $q\leq j \leq n$,
	where $\chi'(\varrho)\geq C_3$ on $X_v\setminus \overline{X}_u$ in Lemma \ref{keylem} and the $L^2$-norm is given by $\omega$, $h^{L^k}_\chi$ and $h^E$ on $X_c$.
\end{prop}    
\begin{proof}
	We follow \cite[Theorem 3.5.8]{MM}.
	Since $B^{0,j}(X_c,L^k\otimes E)$ is dense in $\Dom(\ddbar^E_k)\cap \Dom(\ddbar^{E*}_{k,H})\cap L^2_{0,j}(X_c,L^k\otimes E)$ with respect to the graph norm of $\ddbar^E_k+\ddbar^{E*}_{k,H}$, we only to show this inequality holds for $s\in B^{0,j}(X_c,L^k\otimes E)$ with $j\geq q$ and large $k$.
	
	Suppose now $s\in B^{0,j}(X_c,L^k\otimes E)$. Let $\phi$ be in Lemma \ref{K'lem}. Thus $\phi s \in B^{0,j}(X_c,L^k\otimes E)$ with $\supp(\phi s)\subset X_v\setminus \ov{X}_u$. By Lemma \ref{keylem}, there exists $C_2>0$ and $C_3>0$ such that for $j\geq q$ and $k\geq 1$, we have
	\begin{equation}
	\|\phi s\|^2\leq \frac{C_2}{k}( \|\ddbar^E_k (\phi s)\|^2+\|\ddbar^{E*}_k (\phi s)\|^2 )
	\end{equation} 
	where $\chi'(\varrho)\geq C_3$ on $X_v\setminus \overline{X}_u$ and the $L^2$-norm $\|\cdot\|$ is given by $\omega$, $h^{L^k}_\chi$ and $h^E$ on $X_c$. Next applying (\ref{cuteq_k}) and Lemma \ref{K'lem}, we obtain
	\begin{equation}
	\|s\|^2-\int_{K'}|s|^2dv_X\leq \frac{5C_2}{k}(\|\ddbar^E_k s\|^2+\|\ddbar^{E*}_k s\|^2)+\frac{12C_1C_2}{k}\|s\|^2.
	\end{equation}
	For $k\geq 24C_1C_2$, it follows that $1-\frac{12C_1C_2}{k}\geq \frac{1}{2}$ and 
	\begin{equation}
	\|s\|^2\leq \frac{10C_2}{k}(\|\ddbar^E_k s\|^2+\|\ddbar^{E*}_k s\|^2)+2\int_{K'}|s|^2 dv_X.
	\end{equation}
	The proof is comlete by choosing $C_0:=\max\{10C_2,2\}$ and $k\geq 24C_1C_2$.
\end{proof}  

Thirdly, we will show that $(L_\chi, h^{L_\chi})$
is semi-positive if $(L, h^L)$ is semipositive by choosing a appropriate $\chi$. 
Let $C_2>0$ and $C_3>0$ be in Lemma \ref{keylem}. 
We choose $\chi\in\cC^\infty(\R)$ such that $\chi''(t)\geq 0$, 
$\chi'(t)\geq C_3$ on $(u,v)$ and $\chi(t)=0$ on $(-\infty,u_0)$. 
Therefore, $\chi'(\varrho(x))\geq C_3>0$ on $X_v\setminus \overline{X}_u$ 
and $\chi(\varrho(x))=\chi'(\varrho(x))=0$ on $X_{u_0}$. 
Note that $K\subset X_{u_0}$ and $u_0<u<c<v$. 
Now we have a fixed $\chi$ which leads to the following proposition.
\begin{prop}\label{pospos}
	Let $X$ be a $q$-convex manifold with $\sqrt{-1}\dbar\ddbar\varrho\geq 0$ on $X_c\setminus K$. If $(L,h^L)$ is semipositive on $X_c$, then $(L,h^{L}_\chi)$ is semipositive on $X_c$ for $\chi$ defined above.
\end{prop}
\begin{proof}
	From the above definition of $\chi$, we have $\chi'(\varrho)\geq 0$ on $X$, $\chi'(\varrho)=0$ on $K$. Since $\varrho$ is plurisubharmonic on $X_c\setminus K$, i.e., $\imat\dbar\ddbar\varrho\geq 0$ on $X_c\setminus K$, we have $\imat \chi'(\varrho)\dbar\ddbar\varrho\geq 0$ on $X_c$. Since $\chi''(\varrho)\geq 0$ and  $\imat\dbar\varrho\wedge\ddbar\varrho\geq 0$ on $X_c$, we have $\imat \chi''(\varrho)\dbar\varrho\wedge\ddbar\varrho\geq 0$ on $X_c$. Finally $\imat R^{L_\chi}=\imat R^L+\imat \chi'(\varrho)\dbar\ddbar\varrho+\imat \chi''(\varrho)\dbar\varrho\wedge\ddbar\varrho\geq 0$ on $X_c$.
\end{proof}

Now we can prove the main result of this section as follows.
 
\begin{proof}[Proof of Theorem \ref{qconvexpshthm}]
By Proposition \ref{pospos} with the fixed $\chi$, 
Proposition \ref{1coxFE} and using Theorem \ref{L2FE} 
for $X_c$ endowed with Hermitian metric $\omega$ obtained in 
Lemma \ref{lowbd_rho_lem}, there exists $C>0$ 
such that for any $j\geq q$ and sufficiently large $k$,
\begin{equation}
\dim H_{(2)}^{0,j}(X_c,L^k\otimes E)=\dim \cH^{0,j}(X_c,L^k\otimes E)\leq Ck^{n-j}
\end{equation}
holds with respect to the metrics $\omega$, $h^{L}_\chi$ and $h^E$ on $X_c$.
From \cite[Theorem 3.5.6 (H\"{o}rmander), Theorem 3.5.7 (Andreotti-Grauert)(i), Theorem B.4.4 (The Dolbeault isomorphism)]{MM}, we have for $j\geq q$,
$$H^j(X,L^k\otimes E)\cong H^j(X_v,L^k\otimes E)\cong 
H^{0,j}(X_v,L^k\otimes E)\cong H_{(2)}^{0,j}(X_c,L^k\otimes E).$$ 
Thus the conclusion holds for sufficiently large $k$. 
Also we know that for any holomorphic vector bundle $F$, 
$\dim H^j(X,F)<\infty$ for $j\geq q$ by 
\cite[Theorem B.4.8 (Andreotti-Grauert)]{MM}. 
So the conclusion holds for all $k\geq 1$.
\end{proof}
 
\begin{proof}[Proof of Theorem \ref{T:1c}] 
Apply Theorem \ref{qconvexpshthm} for $1$-convex manifolds.
\end{proof} 

By adapting the duality formula \cite[20.7 Theorem]{HL:88} 
for cohomology groups to Theorem \ref{qconvexpshthm}, 
we have the analogue result to \cite[Remark 4.4]{Wh:16} 
for seminegative line bundles.
\begin{cor}
	Let $X$ be a $q$-convex manifold of dimension $n$ with a plurisubharmonic exhaustion function $\varrho$ near the exceptional set $K$, and let $(L,h^L), (E,h^E)$ be holomorphic Hermitian line bundles on $X$. Let $(L,h^L)$ be seminegative on a sublevel set $X_c=\{x\in X: \varrho(x)<c \}$ satisfying $K\subset X_c$. Then there exists $C>0$ such that for any $0\leq j\leq n-q$ and  $k\geq 1$, the $j$-th cohomology with compact supports
	\begin{equation}
	\dim [H^{0,j}(X,L^k\otimes E)]_0\leq Ck^{j}.
	\end{equation}
	In particular, 
	$\dim [H^{0,0}(X,L^k\otimes E)]_0\leq C$.
\end{cor}
\begin{proof}
	Combine \cite[20.7 Theorem]{HL:88} and Theorem \ref{qconvexpshthm}.
\end{proof}

\subsubsection{Line bundles with positivity near the exceptional set}
A natural question in Theorem \ref{qconvexpshthm} is whether the hypothesis on plurisubharmonic exhaustion function is necessary. In this section, we show that such hypothesis can be replaced by the positivity assumption of line bundle near the exceptional set.

\begin{proof}[Proof of Theorem \ref{qconvexpos}]
	Let $X$ be a $q$-convex manifold of dimension $n$, let $(L,h^L), (E,h^E)$ be holomorphic Hermitian line bundles on $X$. Let $\varrho\in \cC^\infty(X,\R)$ be the exhaustion function satisfying $\sqrt{-1}\dbar\ddbar\varrho$ has $n-q+1$ positive eigenvalues on $X\setminus K$.
	
	Since $(L,h^L)$ be semipositive on a sublevel set $X_c=\{x\in X: \varrho(x)<c \}$ satisfying the exceptional set $K\subset X_c$ and positive on $X_c\setminus K$, we fix real numbers $u_0, u, c'$ and $v$ satisfying $u_0<u<c'<v<c$ and $K\subset X_{u_0}$. Thus for $X_{c'}:=\{ x\in X: \varrho(x)<c' \}$ such that $K\subset X_{c'}$ and $L\geq 0$ on $X_{c'}$ and $L>0$ on $X_v\setminus K$. For simplifying notions, we still denote $c'$ by $c$ in this proof, that is, there exists real numbers $u_0, u, c$ and $v$ satisfying $u_0<u<c<v$ and $K\subset X_{u_0}$ and $L\geq 0$ on $X_c$ and $L>0$ on $X_v\setminus K$.
	
	Firstly, we choose the metric $\omega$ on $X$ from Lemma \ref{lowbd_rho_lem}. 
	
	Secondly, we show the fundamental estimate holds. Note $L>0$ on $X_v\setminus K$, by the same argument in Lemma \ref{keylem} without modification of $h^L$ by $h^L_\chi$, we have that
	there exists $C_2>0$ and $k_0>0$  such that
	\begin{equation}\label{eqfe1}
	\|s\|^2\leq \frac{C_2}{k}( \|\ddbar^E_k s\|^2+\|\ddbar^{E*}_k s\|^2 )
	\end{equation}    
	for $s\in B^{0,j}(X_c,L^k\otimes E)$ with $\supp(s)\subset X_v\setminus \overline{X}_u$, $j\geq q$ and $k\geq k_0>0$,
	where the $L^2$-norm $\|\cdot\|$ is given by $\omega$, $h^{L^k}$ and $h^E$ on $X_c$. As in Lemma \ref{1coxFE} without the modification of $h^L$, we conclude that
	there exist a compact subset $K'\subset X_c$ (In fact, let $\epsilon>0$ such that $\{ c-\epsilon<\varrho<c+\epsilon \}\Subset X_v\setminus \ov{X}_u$, $K':=\ov{\{ \varrho<c-\epsilon \}}$ as in Lemma \ref{K'lem}) and $C_0>0$ such that for sufficiently large $k$, we have 
	\begin{equation}
	\|s\|^2\leq \frac{C_0}{k}(\|\ddbar^E_ks\|^2+\|\ddbar^{E*}_{k,H} s\|^2)+C_0\int_{K'} |s|^2 dv_X
	\end{equation}
	for any $s\in \Dom(\ddbar^E_k)\cap \Dom(\ddbar^{E*}_{k,H})\cap L^2_{0,j}(X_c,L^k\otimes E)$ and each $q\leq j \leq n$,
	where the $L^2$-norm is given by $\omega$, $h^{L^k}$ and $h^E$ on $X_c$. 
	 
	Finally, we can apply Theorem \ref{L2FE} on $X_c$. Therefore, there exists $C>0$ such that for any $j\geq q$ and sufficiently large $k$,
	\begin{equation}
	\dim H_{(2)}^{0,j}(X_c,L^k\otimes E)\leq Ck^{n-j}
	\end{equation}
	holds with respect to the metrics $\omega$, $h^{L}$ and $h^E$ on $X_c$.   
	As in the proof of Theorem \ref{qconvexpshthm}, the conclusion holds for all $k\geq 1$.
\end{proof}	
  
By adapting the duality formula \cite[20.7 Theorem]{HL:88} for cohomology groups to Theorem \ref{qconvexpos}, we have the analogue result to \cite[Remark 4.4]{Wh:16} for seminegative line bundles.
\begin{cor}
	Let $X$ be a $q$-convex manifold of dimension $n$, and let $(L,h^L), (E,h^E)$ be holomorphic Hermitian line bundles on $X$. Let $\varrho$ be an exhaustion function of $X$ and $K$ the exceptional set. Let $(L,h^L)$ be seminegative on a sublevel set $X_c=\{x\in X: \varrho(x)<c \}$ satisfying $K\subset X_c$, and let $(L,h^L)$ be negative on $X_c\setminus K$.	
	Then there exists $C>0$ such that for any $0\leq j\leq n-q$ and  $k\geq 1$, the $j$-th cohomology with compact supports
	\begin{equation}
	\dim [H^{0,j}(X,L^k\otimes E)]_0\leq Ck^{j}.
	\end{equation}
	In particular, 
	$\dim [H^{0,0}(X,L^k\otimes E)]_0\leq C$.
\end{cor}
\begin{proof}
	Combining \cite[20.7 Theorem]{HL:88} and Theorem \ref{qconvexpos}.
\end{proof}

\begin{rem}[Compatibility to the vanishing theorem on $q$-convex manifolds] \label{remqconvex}
	Let $(E,h^E)$ be a holomorphic vector bundle on $X$. If $(L,h^L)>0$ on $X_c$ with $K\subset X_c$ instead of the hypothesis $(L,h^L)\geq 0$ on $X_c\setminus K$ in Theorem \ref{qconvexpos}, then for $j\geq q$ and sufficiently large $k$,
	\begin{equation}
	\dim H^j(X,L^k\otimes E)=0.
	\end{equation}
 In particular, if $X$ is $q$-complete, then $H^p(X,E)=0$ for $p\geq q$ and arbitrary holomorphic vector bundle $E$. In fact, the fundamental estimates hold with $K=\emptyset$ in these cases, thus the space of harmonic forms is trivial, see \cite[Theorem 3.5.9]{MM}.
\end{rem}   	
   
\subsection{Pseudo-convex domains}
\begin{thm}
		Let $M\Subset X$ be a smooth strongly pseudoconvex domain in a complex manifold $X$ of dimension $n$. Let $(L,h^L)$ and $(E,h^E)$ be holomorphic Hermitian line bundles on $X$. Let $(L,h^L)$ be semipositive on $M$. 
	Then there exists $C>0$ such that for any $q\geq 1$ and $k\geq 1$, we have
	\begin{equation}
	\dim H^q(M,L^k\otimes E)\leq Ck^{n-q}.  
	\end{equation}
\end{thm} 
\begin{proof}
	Note a strongly pseudocovex domain is $1$-convex and applying Theorem \ref{T:1c}.
\end{proof}
  
\begin{proof}[Proof of Theorem \ref{pseudoconvex}]
	We follow \cite[Theorem 3.5.10]{MM}.
	Let $\omega$ be a Hermitian metric on $X$. 
  Note $L>0$ around $bM$, by the same argument in Lemma \ref{1coxFE} without the modification of $h^L$, we conclude that
  there exist a compact subset $K'\subset M$ and $C_0>0$ such that for sufficiently large $k$,
  \begin{equation}
  \|s\|^2\leq \frac{C_0}{k}(\|\ddbar^E_ks\|^2+\|\ddbar^{E*}_{k,H} s\|^2)+C_0\int_{K'} |s|^2 dv_X
  \end{equation}
  for any $s\in \Dom(\ddbar^E_k)\cap \Dom(\ddbar^{E*}_{k,H})\cap L^2_{0,q}(M,L^k\otimes E)$ and each $1\leq q \leq n$,
  where the $L^2$-norm is given by $\omega$, $h^{L^k}$ and $h^E$ on $M$. Finally, we apply Theorem \ref{L2FE} on $M$.
\end{proof} 
 
\subsection{Weakly $1$-complete manifolds}
\begin{proof}[Proof of Theorem \ref{weakly1complete}]
	We follow \cite[Theorem 3.5.12]{MM}.
	Let $\varphi\in \cC^\infty(X,\R)$ be an exhaustion function of $X$ such that $\imat\dbar\ddbar\varphi\geq 0$ on $X$ and $X_c:=\{\varphi<c\}\Subset X$ for all $c\in \R$. We choose  a regular value $c\in \R$ of $\varphi$ such that $K\subset X_c$ by  Sard's theorem. Thus $X_c$ is a smooth pseudoconvex domain and $L>0$ on a neighbourhood of $bX_c$. We apply Theorem \ref{pseudoconvex}, for any $q\geq 1$ and sufficiently large $k$
	\begin{equation}
	\dim H^{0,q}_{(2)}(X_c,L^k\otimes E)=\dim \cH^{0,q}(X_c,L^k\otimes E)\leq Ck^{n-q}.
	\end{equation}
	Finally, by \cite[Theorem 6.2]{Takegoshi:83} (see \cite[Theorem 3.5.11]{MM}) and Dolbeault isomorphism, it follows that
	$
	H^{q}(X,L^k\otimes E)\cong H^q(X_c,L^k\otimes E)\cong H^{0,q}(X_c,L^k\otimes E)\cong\cH^{0,q}(X_c,L^k\otimes E)
	$ for $q\geq 1$ and sufficiently large $k$.
\end{proof}

\begin{rem}
Marinescu \cite{M:92} positively answered a question of Ohsawa \cite[\S1. Remark 2]{Oh:82} by proving Morse inequalities on weakly $1$-complete manifolds. If $L$ is $q$-positive outside a compact subset $K\subset X_c$, $\dim H^p(X_c,L^k)$ are at most of polynomial growth of degree $n$ with respect to $k$ for $p\geq q$. Theorem \ref{weakly1complete} says that if $L$ is $1$-positive outside a compact subset $K\subset X_c$ and $L\geq 0$ on $X_c$ additionally, then
 $\dim H^p(X_c,L^k)$ are at most of polynomial growth of degree $n-p$ with respect to $k$ for $p\geq 1$.
\end{rem}  
 
Generally we have the following result when $(L,h^L)$ might be not semipositive. 
\begin{cor}\label{weak1cor}
	Let $X$ be a weakly $1$-complete manifold of dimension $n$. Let $(L,h^L)$ and $(E,h^E)$ be holomorphic Hermitian line bundles on $X$. Suppose there exists $f\in \cC^\infty(X,\R)$ such that $\sqrt{-1}R^{(L,h^L)}+\sqrt{-1}\dbar\ddbar f\geq 0$ on $X$, and $\sqrt{-1}R^{(L,h^L)}+\sqrt{-1}\dbar\ddbar f> 0$ on $X\setminus K$ for a compact subset $K\subset X$.  
	Then there exists $C>0$ such that for any $q\geq 1$ and sufficiently large $k$, we have
	\begin{equation}
	\dim H^q(X,L^k\otimes E)\leq Ck^{n-q}.
	\end{equation}
\end{cor} 
\begin{proof}
	Apply theorem \ref{weakly1complete} for the line bundle $(L,h^Le^{-f})$.
\end{proof} 
Note that, by definition, there always exists a smooth function $f$ 
on a weakly $1$-convex manifold $X$ such that 
$\sqrt{-1}\dbar\ddbar f\geq 0$, 
thus Corollary \ref{weak1cor} implies the estimate 
may still hold when $(L,h^L)$ is not semipositive everywhere.	
 
 \begin{rem}(Compatibility to Nakano vanishing theorem on weakly $1$-convex manifolds \cite[Theorem 3.5.15]{MM}) If $(L,h^L)>0$ on $X$ instead of $(L,h^L)\geq 0$ in Theorem \ref{weakly1complete}, then for $q\geq 1$ and sufficiently large $k$, $H^q(X,L^k\otimes E)=0$, which implies the vanishing theorem on $1$-convex (in particular, compact) manifold in Remark \ref{remqconvex}.
 \end{rem}
 
\subsection{Complete manifolds}
  
\begin{proof}[Proof of Theorem \ref{complete}]  
We follow \cite[Theorem 3.3.5]{MM}.	 
Since $(X,\omega)$ is complete, $\ddbar^{E*}_{k,H}=\ddbar^{E*}_k$, that is the Hilbert space adjoint and the maximal extension of the formal adjoint of $\ddbar^E_k$ coincide. Let $ \Lambda=i(\omega)$ be the adjoint of the operator $\omega\wedge\cdot$ with respect to the Hermitian inner product induced by $\omega$, $h^L$ and $h^E$. In a local orthonormal frame $\{ \omega_j \}_{j=1}^n$ of $T^{(1,0)}X$ with dual frame $\{ w^j\}_{j=1}^n$ of $T^{(1,0)*}X$, $\omega=\sqrt{-1}\sum_{j=1}^n \omega^j\wedge\ov\omega^j$ and $\Lambda=-\sqrt{-1}i_{\ov w_j}i_{w_j}$. Thus $\sqrt{-1}R^{(L,h^L)}=\sqrt{-1}\sum_{j=1}^n \omega^j\wedge\ov\omega^j$ outside $K$.

Let $F:=E\otimes K_X^*$ with Hermitian metric $h^F$ induced by $h^E$ and $\omega$. Let $\{e^F_k\}$ be a local frame of $L^k\otimes F$. For $s\in \Omega^{0,q}_0(X\setminus K,L^k\otimes E)=\Omega^{n,q}_0(X\setminus K,L^k\otimes F)$, we can write $s=\sum_{|J|=q} s_J\omega^1\wedge\cdots\wedge\omega^n\wedge\ov\omega^J\otimes e_k^F$ locally, thus
\begin{equation}
[\sqrt{-1}R^L,\Lambda]s
=\sum_{|J|=q}(q s_J\omega^1\wedge\cdots\wedge\omega^n\wedge\ov\omega^J)\otimes e^F_k 
= qs.
\end{equation}
Since $(X\setminus K, \sqrt{-1}R^{(L,h^L)})$ is K\"{a}hler, we apply Nakano's inequality \cite[(1.4.52)]{MM}, 
\begin{equation}
\|\ddbar_k^E s\|^2+\|\ddbar^{E*}_k s\|^2=
	\langle \square^{L^k\otimes E}s,s \rangle 
	\geq k \langle  [\imat R^L,\Lambda]s,s \rangle+[\imat R^E,\Lambda]s,s \rangle.
\end{equation}
Thus there exists $C_E>0$ such that, for any $1\leq q \leq n$,
\begin{equation}
\|\ddbar_k^E s\|^2+\|\ddbar^{E*}_k s\|^2\geq ( qk-C_E)\|s\|^2\geq (k-C_E)\|s\|^2.
\end{equation} 
Therefore, we have
\begin{equation}
  \|s\|^2\leq \frac{2}{k}( \|\ddbar^E_k s\|^2+\|\ddbar^{E*}_k s\|^2 ),
\end{equation}
for $s\in \Omega^{0,q}_0(X\setminus K,L^k\otimes E)$ with $1\leq q\leq n$ and $k\geq 2 C_E$.

Next we follow the analogue argument in Proposition \ref{1coxFE} to obtain the fundamental estimates as follows. Let $V$ and $U$ be open subsets of $X$ such that $K\subset V\Subset U\Subset X$. We choose a function $\xi\in \cC^\infty_0(U,\R)$ such that $0\leq \xi\leq 1$ and $\xi\equiv 1$ on $\ov V$. We set $\phi:=1-\xi$, thus $\phi\in \cC^\infty(X,\R)$ satisfying $0\leq \phi \leq 1$ and $\phi \equiv 0$ on $\ov V$. 
		
	Now let $s\in \Omega_0^{0,q}(X,L^k\otimes E)$, thus $\phi s\in \Omega^{0,q}_0(X\setminus K, L^k\otimes E)$. We set $K':=\ov U$, then 
	\begin{equation}
	\|\phi s\|^2\geq \|s\|^2-\int_{K'}|s|^2dv_X,
	\end{equation}
and similarly there exists a constant $C_1>0$ such that	
\begin{equation}
\frac{1}{k}(\|\ddbar^E_k (\phi s)\|^2+\|\ddbar^{E*}_k(\phi s)\|^2)\leq \frac{5}{k}(\|\ddbar^E_k s\|^2+\|\ddbar^{E*}_k s\|^2)+\frac{12C_1}{k}\|s\|^2.
\end{equation}	

By combining the above three inequalities, there exists $C_0>0$ such that for any $s\in \Omega^{0,q}_0(X,L^k\otimes E)$ with $1\leq q\leq n$ and $k$ large enough
\begin{equation}   
\|s\|^2\leq \frac{C_0}{k}(\|\ddbar^E_ks\|^2+\|\ddbar^{E*}_ks\|^2)+C_0\int_{K'} |s|^2 dv_X.
\end{equation}	
Finally, since $\Omega_0^{0,\bullet}(X,L^k\otimes E)$ is dense in $\Dom(\ddbar_k^E)\cap \Dom(\ddbar_k^{E*})$ in the graph-norm, for each $1\leq q\leq n$ the fundamental estimate holds in bidegree $(0,q)$ for forms with values in $L^k\otimes E$ for $k$ large.
So the conclusion follows from Theorem \ref{L2FE}.
\end{proof}

\subsection{Compact manifolds and coverings revisited}
\begin{thm}[\cite{BB:02}]\label{cmptthm}
	Let $X$ be a compact complex manifold of dimension $n$. Let $(L,h^L)$ and $(E,h^E)$ be holomorphic Hermitian line bundles on $X$. Assume $(L,h^L)$ is semipositive on $X$. Then there exists $C>0$ such that for any 
	$q\geq 1$ and  $k\geq 1$ we have
	\begin{equation}
	\dim H^{0,q}(X, L^k\otimes E)
	\leq C k^{n-q}.
	\end{equation}
\end{thm}
Note that this theorem is a special case of Theorem \ref{L2FE}, \ref{complete}, \ref{qconvexpos}, \ref{T:1c}, \ref{weakly1complete}, \ref{singlinebundle}
 and the following Theorem \ref{coveringthm}.

\begin{cor}\label{cmpcorneg}
	Let $X$ be a compact manifold of dimension $n$, and let $(L,h^L), (E,h^E)$ be holomorphic Hermitian line bundles on $X$. Let $(L,h^L)$ be seminegative on $X$.
	Then there exists $C>0$ such that $0\leq q\leq n-1$,
	\begin{equation}
	\dim H^{0,q}(X,L^k\otimes E)\leq Ck^{q}.
	\end{equation}
	In particular, 
	$\dim H^{0,0}(X,L^k\otimes E)\leq C$.
\end{cor}	
\begin{proof}
	It follows from Serre duality and Theorem\ref{cmptthm}.
\end{proof}

For the case of nef line bundles, the following observation refines the estimates in Theorem\ref{cmptthm} as well as Corollary \ref{cmpcorneg}, and reflects that the magnitude $k^{n-q}$ are precise. 
 
\begin{lemma}[\cite{DPS:94}]\label{DPSlem}
Let $L$ be a nef holomorphic line bundle on a compact complex manifold $X$. Then every non-trivial section in $H^{0}(X,L^{-1})$ has no zero at all.	
\end{lemma}

\begin{cor}\label{cmpcorfinit}
		Let $X$ be a compact manifold of dimension $n$, and let $L$ be holomorphic Hermitian line bundles on $X$. Let $L$ be nef.
	Then, for any $k\in \N$,
	\begin{equation}
	 0\leq \dim H^{0,n}(X,L^k\otimes K_X)=\dim H^{0,0}(X,L^{-k})\leq 1. 
	\end{equation} 
\end{cor}
\begin{proof}
	Suppose there exists $k\geq 1$ such that 
	$H^0(X,L^{-k})\neq 0$, then there exists $s_0\in H^0(X,L^{-k})$ 
	such that $s_0(x)\neq 0$, $\forall x\in X$ by Lemma \ref{DPSlem}. 
	Let $s\in H^0(X,L^{-k})$. It follows that $s\otimes s_0^{-1}\in \cO(X)=\C$, 
	thus $H^0(X,L^{-k})=\C s_0$. And the case of $k=0$ is trivial.	
\end{proof} 

\begin{rem} 
	Let $X$ be a compact complex manifold of dimension $n$ 
	and $(L,h^L)$ a holomorphic Hermitian line bundle on $X$. 
	Let $M:=\{ v\in L^*: |v|_{h^{L^*}}=1 \}$. 
	It is known that the $\ddbar_b$ (Kohn-Rossi) cohomology 
	$H^q_{b,k}(M)\cong H^q(X,L^k)$, see 
	\cite[Section 1.5]{MM:06}, \cite[(2.8)]{HsiaoLi:16}. 
	Thus if $L\geq 0$ on $X$, $\dim H^q_{b,k}(M)\leq Ck^{n-q}$ for all $q\geq 1$ and $k\geq 1$.
\end{rem}
 
The study of $L^2$ cohomology spaces on coverings of compact manifolds
has also interesting applications, cf.\ \cite{GHS:98, Ko:95}.
The results are similar to the case of compact manifolds, but
we have to use the reduced $L^2$ cohomology groups and von Neumann dimension 
instead of the usual dimension, see \cite[Theorem 1.2]{Wh:16} or Theorem \ref{coveringthm}. For example, in the situation of Theorem \ref{coveringthm}, if the line bundle $(L,h^L)$
is positive, the Andreotti-Vesentini vanishing theorem \cite{AV:65} 
shows that ${\overline{H}}^{0,q}_{(2)}(X, L^k\otimes E)\cong\cH^{0,q}(X, L^k\otimes E)=0$ 
for $q\geq1$ and $k$ large enough. The holomorphic Morse inequalities of Demailly \cite{Dem:85} were generalized to coverings
by Chiose-Marinescu-Todor \cite{MTC:02, TCM:01} (cf.\ also \cite[(3.6.24)]{MM}) and yield in the conditions
of Theorem \ref{coveringthm} that $\dim_{\Gamma} {\overline{H}}^{0,q}_{(2)}(X,L^k\otimes E)=o(k^n)$ 
as $k\to\infty$ for $q\geq1$. Hence Theorem \ref{coveringthm} generalizes \cite{BB:02} 
to covering manifolds and refines the estimates obtained in \cite{MTC:02, TCM:01}. 

\begin{thm}[\cite{Wh:16}]\label{coveringthm}
	Let $(X,\omega)$ be a $\Gamma$-covering manifold of dimension $n$.
	Let $(L,h^L)$ and $(E,h^E)$ be two $\Gamma$-invariant 
	holomorphic Hermitian line bundles on $X$. Assume $(L,h^L)$ is semipositive on $X$.
	Then there exists $C>0$ such that for any 
	$q\geq 1$ and  $k\geq 1$ we have
	\begin{equation}
	\dim_{\Gamma}{\overline{H}}^{0,q}_{(2)}(X, L^k\otimes E)=\dim_{\Gamma}\cH^{0,q}(X, L^k\otimes E)
	\leq C k^{n-q}.
	\end{equation}
\end{thm}   
See \cite[Theorem 1.2]{Wh:16} for the complete proof. As a remark, note that for a fundamental domain $U\Subset X$ with respect to $\Gamma$, we used
	\begin{equation}
	\dim_{\Gamma}\cH^{0,q}(X,L^k\otimes E) 
	=\int_{U}B^q_k(x) d v_{X},
	\end{equation} 
	which is similar to the formula
	\begin{equation}
	\dim \cH^{0,q}(X,L^k\otimes E)\leq C_0\int_K B^q_k(x)dv_X
	\end{equation}
	 used in the proof of Theorem \ref{L2FE}. 
 
 The following estimate was firstly obtained by Morse inequalities for covering manifolds \cite{MTC:02,TCM:01}. We do not assume the holomorphic Morse inequalities for coverings in our proof.
\begin{cor}\label{coveringpos} Under the hypothesis of Theorem \ref{coveringthm},
		if $(L,h^L)$ is positive at least at one point additionally, then there exists $C_1>0$ and $C_2>0$ such that for $k$ large enough
		\begin{equation}
		C_1k^n \leq \dim_\Gamma H^{0}(X, L^k\otimes E)\leq C_2k^n.
		\end{equation}
\end{cor}  
   
\begin{proof}
	 It follows from Riemann-Roch-Hirzenbruch formula for coverings \cite{Atiyah:76}
	\begin{equation}
	\sum_{q=0}^{n}(-1)^{n-q}\dim_{\Gamma}\ov H^{0,q}_{(2)}(X,L^k\otimes E)=\frac{k^n}{n!}\int_{X/\Gamma}(-1)^{n}c_1(L/\Gamma,h^{L/\Gamma})^n+o(k^n)
	\end{equation}
	(see also \cite[Theorem 3.6.7]{MM}) and Theorem \ref{coveringthm}.
\end{proof} 

\section{Analytic singular line bundles and vanishing theorems}\label{sectionsing}
   
Let $X$ be a connected compact complex manifold of dimension $n$ and $L$ a holomorphic line bundle on $X$. Let $h^L$ be a Hermitian metric with analytic singularities with local weight
\begin{equation} \label{analyticsingfct}
\varphi=\frac{c}{2}\log(\sum_{j\in J}|f_j|^2)+\psi,
\end{equation}
where $J$ is at most countable, $c$ is a non-negative rational number, $f_j$ are non-zero holomorphic functions and $\psi$ is a smooth function, such that $h^L(e_L,e_L)=e^{-2\varphi}$ for a local holomorphic frame $e_L$ of $L$.
We denote by $\mathscr{J}(h^L):=\mathscr{J}(\varphi)$ the Nadel multiplier ideal sheaf of $h^L$. We define the regular part of $X$ with respect to $h^L$ by $R(h^L):=\{x\in X: h^L~\mbox{smooth on a neighborhood of}~x \}$ and the singular part by $S(h^L):=X\setminus R(h^L)$.   
    
\subsection{Semipositive line bundles endowed metric with analytic singularities}
  
In this section, we follow the argument of Bonavero's singular holomorphic Morse inequalities 
\cite{Bona:98} closely and provide three lemmas, see \cite[2.3.2]{MM} for details. 
In the end, we combine them to prove a result which is analogue to Theorem \ref{cmptthm}. 
  
Firstly, we blow up the singularities of $h^L$ as below, 
see \cite[Lemma 2.3.19]{MM}.	   
\begin{lemma}\label{blowup}
There exists a proper modification $\wi{\pi}:\wi{X}\longrightarrow X$ 
such that the local weight $\wi{\varphi}$ of the metric $h^{\wi{L}} =\wi{\pi}^*h^L$
on $\wi{L}=\wi{\pi}^*L$ has the form 
$\varphi\circ\wi{\pi}=\frac{c}{2}\log|g|^2+\wi{\psi}$, 
where $g$ is holomorphic and $\wi{\psi}$ is smooth.
\end{lemma}
Secondly, we construct a smooth hermitian metric 
$h^{\widehat{L}}$ on a modified line bundle $\widehat{L}$ on
$\wi{X}$, see \cite[Lemma 2.3.20, Lemma 2.3.21]{MM}. 
For a holomorphic vector bundle $F$ over $X$, 
we denote  by $\wi F:=\pi^* F$ the pull-back on $\wi X$.
\begin{lemma}\label{modimetric}	
With the proper modification in Lemma \ref{blowup}, there exists a holomorphic line bundel $\widehat{L}$ on $\wi{X}$ and a smooth Hermtian metric $h^{\widehat{L}}$ satisfying the following conditions:
\begin{itemize}
 \item[(1)] There exists $m\in\N\setminus\{0\}$ such that the curvature is locally given by
 \begin{equation}
 R^{(\widehat{L},h^{\widehat{L}})}=2m\dbar\ddbar\wi{\psi};
 \end{equation} 

 \item[(2)] 
 For any $k\in \N\setminus\{0\}$ and arbitrary holomorphic vector bundle $\wi E$ on $\wi X$, there exists $k'\in \N$, $m'\in [ 0,m )$ with $k=mk'+m'$, and a holomorphic vector bundle $\wi{E}_{m'} $ over $\wi{X}$ with $\rank(\wi{E}_{m'})=\rank(\wi{E})$,  such that 
 \begin{equation}\label{relacoheq1}
  H^q(\wi{X},\wi{L}^k\otimes \wi E\otimes \mathscr{J}(h^{\wi L^k}))=H^q(\wi{X},\widehat{L}^{k'}\otimes \wi{E}_{m'}).	
 \end{equation}
Here note that there exists $C_1>0$ and $C_2>0$ such that for any integer $p\in [0,n]$ and $k$ large enough, $C_1k^p\leq k'^p\leq C_2k^p$.
\end{itemize}
\end{lemma}
 
Thirdly, relation to the cohomology on $X$, see \cite[(2.3.45)]{MM}.

\begin{lemma}\label{relacohomology}
	 Let $E$ be an arbitrary holomorphic vector bundle over $X$. With the proper modification in Lemma \ref{blowup}, for all $q\geq 0$ and $k$ large enough, there exists an isomorphism
\begin{equation}\label{relacoheq2}
H^q(X,L^k\otimes E\otimes \mathscr{J}(h^{L^k}))
\cong H^q(\wi{X},\wi{L}^k\otimes \wi{E}\otimes K_{\wi{X}}\otimes \wi{K_X^*}\otimes \mathscr{J}(h^{\wi{L}^k})).
\end{equation}	
\end{lemma}

Finally we substitute these into our setting of semi-positivity of $(L,h^L)$ on $X$, and obtain another generalization of Theorem \ref{cmptthm}.
 
\begin{proof}[Proof of Theorem \ref{singlinebundle}]
	
 Consider a local weight $\varphi$ like in (\ref{analyticsingfct}) defined on an open connected subset $U\subset X$. Thus 
\begin{equation}
S(h^L)\cap U=\{x\in U: \varphi~\mbox{not smooth at}~x \}=\cap_{j\in J} Z(f_j),
\end{equation}
where $Z(f_j):=\{x\in U: f_j(x)=0\}$.
 
Let $\wi{\pi}:\wi{X}\rightarrow X$ be a proper modification
of Lemma \ref{blowup}. The local weight of $(\wi{L},h^{\wi{L}})$ on $\wi{U}:=\wi{\pi}^{-1}(U)$ has the form
\begin{equation}
\varphi\circ\wi{\pi}=\frac{c}{2}{\log(\sum_{j\in J}|f_j\circ\wi{\pi}|^2)}+\psi\circ\wi{\pi}=\frac{c}{2}\log|g|^2+\wi{\psi},
\end{equation}
 where $\wi{\psi}$ is smooth on $\wi{U}$. Thus $Z(g):=\{ z\in \wi U: g(y)=0 \}$ satisfies
 \begin{equation}
  Z(g)=\{ y\in \wi{U}:\varphi\circ\wi{\pi}~\mbox{not smooth at}~y \}=\wi{\pi}^{-1}(S(h^L)\cap U).
 \end{equation}

	Since $c_1(L,h^L)=\frac{\sqrt{-1}}{2\pi}R^{(L,h^L)}\geq 0$ on $X\setminus S(h^L)$, $\varphi$ is smooth plurisubharmonic on $U\setminus S(h^L)$, i.e., $\sqrt{-1}\dbar\ddbar\varphi\geq 0$ on $U\setminus S(h^L)$. Therefore, $\varphi\circ\wi{\pi}$ is smooth plurisubharmonic on $\wi{\pi}^{-1}(U\setminus S(h^L))=\wi{U}\setminus Z(g)$, and since $\dbar\ddbar\log|g|^2=0$ on $\wi U\setminus Z(g)$, we have on $\wi{U}\setminus Z(g)$
	\begin{equation}
	\sqrt{-1}\dbar\ddbar \wi{\psi}=\sqrt{-1}\dbar\ddbar(\varphi\circ \wi{\pi})\geq 0.
	\end{equation}	
  
Next we show $\sqrt{-1}\dbar\ddbar\wi{\psi}\geq 0$ on $Z(g)$. 
In fact, suppose there exists $y_0\in Z(g)$ such that 
$\sqrt{-1}\dbar\ddbar\wi{\psi}(y_0)$ has at least one negative eigenvalue. 
By the smoothness of $\wi\psi$, there exists an open neighbourhood 
$V_0\subset \wi{U}$ of $y_0$ satisfying  $\sqrt{-1}\dbar\ddbar\wi{\psi}$ 
has at least one negative eigenvalue on $V_0$. Since $Z(g)$ 
is nowhere dense subset in $\wi U$, there exists $y_1\in V_0$ such that 
$y_1\in \wi U\setminus Z(g)$. So we obtain a contradiction.

Finally, by using Lemma \ref{modimetric} (1),  we obtain $c_1(\widehat{L},h^{\widehat{L}})=
\frac{\sqrt{-1}}{2\pi}R^{(\widehat{L},h^{\widehat{L}})}=
\frac{m}{\pi}\sqrt{-1}\dbar\ddbar\wi{\psi}\geq 0$ on $\wi{U}$, 
i.e., $(\widehat{L},h^{\widehat{L}})$ is a semipositive line bundle 
on $\wi{X}$. By (\ref{relacoheq2}), (\ref{relacoheq1}) and 
Theorem \ref{cmptthm} applied to $\wi X$ with $(\widehat L,h^{\widehat{L}})$, we have  
$$
\dim H^q(X, L^k\otimes E\otimes\mathscr{J}(h^{L^k}))
=H^q(\wi{X},\widehat{L}^{k'}\otimes (\wi{E}\otimes K_{\wi{X}}\otimes \wi{K_X^*})_{m'})
\leq C k'^{n-q}\leq CC_2k^{n-q}.
$$
Since a positive current $c_1(L,h^L)$ on $X$ is semipositive on $R(h^L)$, the last assertion follows. 
\end{proof}
    
In analogy to the covering manifolds case, we have the following estimate 
of the space of holomorphic sections obtained firstly by Bonavero, 
see \cite{Boucksom:02} and \cite[Corollary 2.3.46]{MM}. 
We remark that we do not use Bonavero's holomorphic Morse inequalities 
in our proof.
\begin{cor}
	Under the hypothesis of Theorem \ref{singlinebundle},
	if $(L,h^L)$ is positive at least at one point additionally, then for $k$ large enough
	\begin{equation}
	C_1k^n\leq \dim H^{0}(X, L^k\otimes E\otimes \mathscr{J}(h^{L^k}))\leq C_2k^n.
	\end{equation}
	In particular, if $X$ is additionally K\"{a}hler, then $X$ is projective.
\end{cor} 
\begin{proof}
It follows from Theorem \ref{singlinebundle} and asymptotic Riemann-Roch-Hirzenbruch formula for $h^L$ with analytic singularities (see \cite[(2.3.45), (2.3.31), (1.7.1)]{MM})
\begin{equation}\label{singRRH}
\sum_{q=0}^{n}(-1)^{n-q}\dim H^q(X,L^k\otimes E\otimes\mathscr{J}(h^{L^k}))=\frac{k^n}{n!}\int_{R(h^L)} (-1)^n c_1(L,h^L)^n+o(k^n),
\end{equation}
where $R(h^L):=\{x\in X: h^L~\mbox{smooth on a neighborhood of}~x \}$ is the regular part of $X$ with respect to $h^L$.
Thus $L$ is big, $X$ is Moishezon. So $X$ is projective when it is K\"{a}hler.
\end{proof}   

\subsection{Vanishing theorems for singular line bundles}
In this section, let $X$ be a compact complex manifold of dimension $n$, let $E$ and $L$ be holomorphic vector bundles over $X$ with $\rank(L)=1$. The key observation in the following Theorem \ref{singAG62} and \ref{singSiu84thm} is that the curvature conditions of line bundles are preserved under the proper modifications (in Lemma \ref{blowup}) as in the proof of Theorem \ref{singlinebundle}, therefore, the proof of Theorem \ref{singAG62} and \ref{singSiu84thm} are analogue to the argument in Theorem \ref{singlinebundle}.

The vanishing theorems for partially positive line bundles \cite[Ch.VII.(5.1)]{Dem}) entails that for arbitrary smooth Hermitian metric $h^L$ on $L$, if $c_1(L,h^L)$ has at least $n-s+1$ positive eigenvalues on $X$, then 
\begin{equation}\label{eqAG}
H^q(X,E\otimes L^k)=0
\end{equation}
for $q\geq s$ and sufficiently large $k$. 
Similarly we obtain the analytic singular version of (\ref{eqAG}) as follows.
 
\begin{proof}[Proof of Theorem \ref{singAG62}]
	Consider a local weight $\varphi$ like in (\ref{analyticsingfct}) defined on an open connected subset $U\subset X$. Thus 
	\begin{equation}
	S(h^L)\cap U=\{x\in U: \varphi~\mbox{not smooth at}~x \}=\cap_{j\in J} Z(f_j),
	\end{equation}
	where $Z(f_j):=\{x\in U: f_j(x)=0\}$.
	
	Let $\wi{\pi}:\wi{X}\rightarrow X$ be a proper modification
	of Lemma \ref{blowup}. Note $\wi\pi:\wi X\setminus \wi B\cong X\setminus B$ is biholomorphic and $B\subset X$ and $\wi B\subset \wi X$ are nowhere dense (closed) analytic subsets.  The local weight of $(\wi{L},h^{\wi{L}})$ on $\wi{U}:=\wi{\pi}^{-1}(U)$ has the form
	\begin{equation}
	\varphi\circ\wi{\pi}=\frac{c}{2}{\log(\sum_{j\in J}|f_j\circ\wi{\pi}|^2)}+\psi\circ\wi{\pi}=\frac{c}{2}\log|g|^2+\wi{\psi},
	\end{equation}
	where $\wi{\psi}$ is smooth on $\wi{U}$. Thus $Z(g):=\{ z\in \wi U: g(y)=0 \}$ satisfies
	\begin{equation}
	Z(g)=\{ y\in \wi{U}:\varphi\circ\wi{\pi}~\mbox{not smooth at}~y \}=\wi{\pi}^{-1}(S(h^L)\cap U).
	\end{equation}
	
	Firstly, we show the number of positive eigenvalues was preserved by $\wi\pi$ on $\wi X$.
	
	Since $\sqrt{-1}\dbar\ddbar\varphi$ has at least $n-s+1$ positive eigenvalues on $U\setminus S$, 
	and
	\begin{equation}
	\sqrt{-1}\dbar\ddbar \wi{\psi}=\sqrt{-1}\dbar\ddbar(\varphi\circ \wi{\pi})\quad\mbox{on}~\wi{U}\setminus Z(g),
	\end{equation}
	we see $\sqrt{-1}\dbar\ddbar\wi \psi$ has at least $n-s+1$ positive eigenvalues on $\wi U\setminus (Z(g)\cup \wi B)$. 
	
	Next we show $\sqrt{-1}\dbar\ddbar\wi \psi$ has at least $n-s+1$ positive eigenvalues on $\wi U\cap (Z(g)\cup \wi B)$, thus on $\wi U$. In fact, suppose $\sqrt{-1}\dbar\ddbar\wi\psi$ has no more than $n-s$ positive eigenvalues at $x_0\in \wi U\cap (Z(g)\cup \wi B)$. By the smoothness of $\wi\psi$, there exists a neighbourhood $V_0\subset \wi U$ of $x_0$ such that $\sqrt{-1}\dbar\ddbar\wi\psi$ has no more than $n-s$ positive eigenvalues on $V_0$. Since $\wi B$ is nowhere dense, there exists $y_0\in V_0\setminus \wi B$. If $y_0$ is not in $Z(g)$, $y_0\in V_0\setminus (\wi B\cup Z(g))\subset \wi U\setminus (\wi B\cup Z(g))$ which leads to a contradiction. If $y_0\in Z(g)$, $y_0\in (V_0\cap Z(g))\setminus \wi B$. There exists a neighbourhood $W_0\subset V_0\setminus \wi B$ of $y_0$ such that $\sqrt{-1}\dbar\ddbar\wi\psi$ has no more than $n-s$ positive eigenvalues on $W_0$. But $W_0$ is not a subset of the nowhere dense set $Z(g)$, thus there exists a point $z_0\in W_0\setminus Z(g)\subset V_0\setminus (\wi B\cup Z(g))\subset \wi U\setminus (\wi B\cup Z(g))$, and $\sqrt{-1}\dbar\ddbar\wi\psi|_{z_0}$ has no more than $n-s$ positive eigenvalues. So it is a contradiction.
	
	By using Lemma \ref{modimetric} (1), $c_1(\widehat{L},h^{\widehat{L}})=\frac{\sqrt{-1}}{2\pi}R^{(\widehat{L},h^{\widehat{L}})}=\frac{m}{\pi}\sqrt{-1}\dbar\ddbar\wi{\psi}$ on $\wi{U}$, $(\widehat{L},h^{\widehat{L}})$ has at least $n-s+1$ positive eigenvalues on $\wi{X}$. By (\ref{relacoheq2}), (\ref{relacoheq1}) and (\ref{eqAG}) applied to $\wi X$ with $(\widehat L,h^{\widehat{L}})$, we have for $q\geq s$ and sufficiently large $k$,
	\begin{equation}
	\dim H^q(X, L^k\otimes E\otimes\mathscr{J}(h^{L^k}))=H^q(\wi{X},\widehat{L}^{k'}\otimes (\wi{E}\otimes K_{\wi{X}}\otimes \wi{K_X^*})_{m'})=0.
	\end{equation}
	$X(t):=\{ x\in R(h^L): c_1(L,h^L)_x~\mbox{is non-degenerated with exactly}~t~\mbox{negative eigenvalues} \}$ and $X(\leq s-1):=\cup_{0\leq t\leq s-1}X(t)$. Finally, the proof is complete by (\ref{singRRH}).
\end{proof} 

\begin{cor}\label{corsemipositve}	
	Let $X$ be a compact manifold of dimension $n$, let $L$ and $E$ be holomorphic line bundles on $X$. Let $h^L$ be the Hermitian metric on $L$ with analytic singularities as in (\ref{analyticsingfct}). Let $0\leq t\leq n$.
	Assume $c_1(L,h^L)\geq 0$ on $R(h^L)$ and $c_1(L,h^L)$ has at least $n-t$ positive eigenvalues on $R(h^L)$. 
	
	Then, as $k\rightarrow \infty$, we have 
	\begin{eqnarray}
	\quad\quad\quad\quad \dim H^q(X, E\otimes L^k\otimes \mathscr{J}(h^{L^k}))=
	\begin{cases}
	\frac{k^n}{n!}\int_{X(0)}c_1(L,h^L)^n+o(k^n),&\mbox{for}~ q=0,\\   
	\mO(k^{n-q}),&\mbox{for}~ 1\leq q\leq t,\\
	0,&\mbox{for}~ q>t.   
	\end{cases}   
	\end{eqnarray}  
\end{cor}
     
\begin{proof}
	The third equality is from Theorem \ref{singAG62}, the second equality is from Theorem \ref{singlinebundle} and the first is given by combining the second equality and (\ref{eqsum}).
\end{proof}

\begin{rem}
	As a remark on Corollary \ref{corsemipositve}, if $t=0$, $(L,h^L)$ is positive on $R(h^L)$; if $t=n$, $(L,h^L)$ is semipositive on $R(h^L)$. The smooth version (i.e., $ h^L~\mbox{is smooth everywhere}$) of Corollary \ref{corsemipositve} 
	are given by \cite{Dem:85}, \cite{BB:02} and \cite[Ch.VII.(5.1)]{Dem} respectively.
\end{rem}

By \cite[Appendix, Theorem 6 and Theorem 3 (ii)]{Rie:71}, 
for a compact complex manifold $X$ and a holomorphic Hermitian line bundle 
$(L,h^L)$ on $X$, if $X$ is Moishezon and $c_1(L,h^L)\geq 0$ on $X$ and 
$c_1(L,h^L)>0$ at least at one point, 
then $H^q(X,K_X\otimes L)=0$
for all $q\geq 1$.	Furthermore, Siu's resolution 
\cite[Theorem 1 (see Page 175)]{Siu:85}) of Grauert-Riemenschneider conjecture
shows that the Moishezon assumption is unnecessary,     
see also \cite[Ch.VII. (3.5) Theorem]{Dem}.
Thus, Siu's vanishing theorem entails that, for arbitrary smooth Hermitian metric $h^L$ on $L$, 
if $c_1(L,h^L)\geq 0$ on $X$ and $c_1(L,h^L)>0$ at least at one point, then 
\begin{equation}\label{eqSiu}
H^q(X,K_X\otimes L)=0
\end{equation}
for $q\geq 1$. Similarly, the analytic singular version of (\ref{eqSiu}) is as follows.
     
\begin{thm}[Theorem \ref{singSiu84}]\label{singSiu84thm}
	Let $X$ be a compact manifold of dimension $n$, let $L$ be a holomorphic line bundle. Let $h^L$ be the Hermitian metric on $L$ with analytic singularities as in (\ref{analyticsingfct}). If $c_1(L,h^L)\geq 0$ on $R(h^L)$ and $c_1(L,h^L)>0$ at least at one point in $R(h^L)$, then 
	\begin{equation}
	H^q(X,K_X\otimes L^{km}\otimes \mathscr{J}(h^{L^{km}}))=0
	\end{equation}
	for $q\geq 1$, $k\geq 1$ and $m\in \N\setminus\{ 0 \}$ satisfying $mc\in \N$ with non-negative rational number $c$ in (\ref{analyticsingfct}). 
\end{thm}	 

\begin{proof}
	Consider a local weight $\varphi$ like in (\ref{analyticsingfct}) defined on an open connected subset $U\subset X$. Thus 
	\begin{equation}
	S(h^L)\cap U=\{x\in U: \varphi~\mbox{not smooth at}~x \}=\cap_{j\in J} Z(f_j),
	\end{equation}
	where $Z(f_j):=\{x\in U: f_j(x)=0\}$.
	
	Let $\wi{\pi}:\wi{X}\rightarrow X$ be a proper modification
	of Lemma \ref{blowup}. Note $\wi\pi: \wi X\setminus \wi B\cong X\setminus B$ is biholomorphic and $B\subset X$ and $\wi B\subset \wi X$ are nowhere dense (closed) analytic subsets. The local weight of $(\wi{L},h^{\wi{L}})$ on $\wi{U}:=\wi{\pi}^{-1}(U)$ has the form
	\begin{equation}
	\varphi\circ\wi{\pi}=\frac{c}{2}{\log(\sum_{j\in J}|f_j\circ\wi{\pi}|^2)}+\psi\circ\wi{\pi}=\frac{c}{2}\log|g|^2+\wi{\psi},
	\end{equation}
	where $\wi{\psi}$ is smooth on $\wi{U}$. Thus $Z(g):=\{ z\in \wi U: g(y)=0 \}$ satisfies
	\begin{equation}
	Z(g)=\{ y\in \wi{U}:\varphi\circ\wi{\pi}~\mbox{not smooth at}~y \}=\wi{\pi}^{-1}(S(h^L)\cap U).
	\end{equation}
	
	Firstly, we have known that the semipositivity was preserved by $\wi \pi$ on $\wi X$ in the proof of Theorem \ref{singlinebundle}, i.e., $\sqrt{-1}\dbar\ddbar\wi{\psi}\geq 0$ on $\wi U$.
	
	Secondly, we show the positivity at some point was preserved by $\wi \pi$. Let $U\subset X$ be a neighbourhood of this point such that $(L,h^L)$ is positive on $U\setminus S(h^L)$. Since $\sqrt{-1}\dbar\ddbar\varphi$ is positive on $U\setminus S(h^L)$, 
	and on $\wi{U}\setminus Z(g)$
	\begin{equation}
	\sqrt{-1}\dbar\ddbar \wi{\psi}=\sqrt{-1}\dbar\ddbar(\varphi\circ \wi{\pi}).
	\end{equation}$\sqrt{-1}\dbar\ddbar\wi \psi$ is positive on $\wi U\setminus (Z(g)\cup \wi B)$ which is not empty, since the nowhere dense of $\wi B$ and the analytic subset $Z(g)\subset \wi U$. 
	
	Finally, Lemma \ref{modimetric}(1) $c_1(\widehat{L},h^{\widehat{L}})=\frac{\sqrt{-1}}{2\pi}R^{(\widehat{L},h^{\widehat{L}})}=\frac{m}{\pi}\sqrt{-1}\dbar\ddbar\wi{\psi}$ on $\wi{U}$ implies $(\widehat{L},h^{\widehat{L}})$ is semipositve on $\wi{X}$ and positive at least at one point. By (\ref{relacoheq2}), (\ref{relacoheq1}) and (\ref{eqSiu}) applied to $\wi X$ with $(\widehat L,h^{\widehat{L}})$, we have for $q\geq 1$ and $k\geq 1$, 
	\begin{equation}
	\dim H^q(X, K_X\otimes L^{km}\otimes\mathscr{J}(h^{L^{km}}))=H^q(\wi X, K_{\wi X}\otimes \wi L^{km}\otimes  \mathscr{J}(h^{\wi L^{km}}))=H^q(\wi{X}, K_{\wi X}\otimes \hat{L}^k)=0,
	\end{equation}
	where $m\in \N\setminus\{ 0 \}$ is arbitrary number satisfying $mc\in \N$ (see \cite[(2.3.27), (2.3.32)]{MM}).
\end{proof}

\begin{cor}\label{corsiu1}
	Under the hypothesis of Theorem \ref{singSiu84thm}, if $c\in \N$ in (\ref{analyticsingfct}) additionally, then
	\begin{equation}
	H^q(X,K_X\otimes L\otimes \mathscr{J}(h^L))=0
	\end{equation}
	for $q\geq 1$. In particular, if $c\in \N$ in (\ref{analyticsingfct}) and $L=K_X^*$ additionally, then
	\begin{equation} 
	H^q(X,\mathscr{J}(h^{K_X^*}))=0
	\end{equation}
	for $q\geq 1$, and the Euler characteristic
	\begin{equation}
	\chi(X,\mathscr{J}(h^{K_X^*}))=\dim H^0(X,\mathscr{J}(h^{K_X^*}))\leq 1.
	\end{equation}
\end{cor} 

\begin{proof}
	Take $m=k=1$ in Theorem \ref{singSiu84thm}.
\end{proof}
                                            
\section*{Acknowledgements} 
The author thanks Professor Xianzhe Dai, Professor George Marinescu for helpful discussion and enlightened comments. This work is partially supported by Albert's Researcher Reunion Grant of University of Cologne.


\end{document}